\documentclass{article}

\usepackage{amsthm}
\usepackage{amsmath}
\usepackage{amsfonts}
\usepackage{amssymb}
\usepackage{enumerate}
\usepackage{graphicx}

\title{Characterizing downwards closed, strongly first order, relativizable dependencies}

\author{Pietro Galliani\\Pietro.Galliani@unibz.it}

\newtheorem{Theorem}{Theorem}[section]
\newtheorem{Proposition}[Theorem]{Proposition}
\newtheorem{Lemma}[Theorem]{Lemma}
\newtheorem{Corollary}[Theorem]{Corollary}

\theoremstyle{definition}
\newtheorem{Definition}[Theorem]{Definition}

\newcommand{\tuple}{\mathbf}
\newcommand{\FO}{\texttt{FO}}
\newcommand{\ESO}{\texttt{ESO}}
\newcommand{\M}{\mathfrak M}
\newcommand{\D}{\mathbf D} 
\newcommand{\DD}{\mathcal D}

\newcommand{\SSd}{\mathcal S}
\newcommand{\E}{\mathbf E}
\newcommand{\EE}{\mathcal E}
\newcommand{\F}{\mathbf F}
\newcommand{\LL}{\mathcal L}
\newcommand{\U}{\mathbf U}
\newcommand{\parts}{\mathcal P}
\newcommand{\NT}{\texttt{nt}}
\newcommand{\Dmax}{\mathbf D_{\textbf{max}}}
\newcommand{\h}{\underline{~~}}
\newcommand{\dom}{\texttt{Dom}}
\newcommand{\fv}{\text{Free}}

\begin{document}
\maketitle

\begin{abstract}
In Team Semantics, a dependency notion is strongly first order if every sentence of the logic obtained by adding the corresponding atoms to First Order Logic is equivalent to some first order sentence. In this work it is shown that all nontrivial dependency atoms that are strongly first order, downwards closed, and relativizable (in the sense that the relativizations of the corresponding atoms with respect to some unary predicate are expressible in terms of them) are definable in terms of constancy atoms. 

Additionally, it is shown that any strongly first order dependency is safe for any family of downwards closed dependencies, in the sense that every sentence of the logic obtained by adding to First Order Logic both the strongly first order dependency and the downwards closed dependencies is equivalent to some sentence of the logic obtained by adding only the downwards closed dependencies. 
\end{abstract}

\section{Introduction}
Team Semantics \cite{hodges97} generalizes Tarskian semantics for First Order Logic by defining the satisfaction relation with respect to sets of assignments, called \emph{teams}, rather with respect to single assignments. Team Semantics was originally developed in order to provide a compositional semantics equivalent to the imperfect-information, game-theoretic semantics of \emph{Independence-Friendly Logic} \cite{hintikkasandu89,hintikka96,mann11}; however, with the work of V\"a\"an\"anen \cite{vaananen07} it became clear that this semantics is a natural generalization of Tarskian semantics, one which -- even putting aside its connections with the theory of databases (see for instance  \cite{kontinen13}) -- greatly extends its expressive capabilities by allowing the study of a far greater range of atoms and operators. 

One of the most peculiar aspects of Team Semantics is the way in which it straddles the boundary between first order and second order: while the syntaxes of the logics based on Team Semantics studied so far have for the most part a solidly first order flavour, in the sense that they involve no \emph{explicit} higher order quantification\footnote{The \emph{majority quantifier} of \cite{durand2015dependence} could arguably be considered an exception to this.}, their expressive power often rises well above the one of First Order Logic, all the way up to Existential Second Order Logic (as in the cases of Dependence Logic \cite{vaananen07} and Independence Logic \cite{gradel13}) or even to full Second Order Logic (as in the case of Team Logic \cite{vaananen07b,kontinen2011team}). 

This is the case even for logics, such as the above mentioned Dependence Logic and Independence Logic or for Inclusion Logic \cite{galliani12}, which only add to the language of First Order Logic \emph{dependency atoms} whose satisfaction conditions are first order definable as properties of relations. In brief, this is due to the higher order quantification hidden in Team Semantics rules for disjunction and existential quantification (see Rules \textbf{TS-$\vee$} and \textbf{TS-$\exists$} of Definition \ref{def:teamsem}). For First Order Logic proper there exists a strict equivalence between Team Semantics and the usual Tarskian semantics on the level of sentences (see Proposition \ref{propo:FOsent}), but this equivalence fails badly on the level of formulas: there exist properties of relations that are first order definable (in the sense that they are defined by a first order sentence $\phi(R)$), but that do not correspond to the satisfaction conditions of any first order \emph{formula} in Team Semantics (that is, there is no first order formula $\psi(\tuple x)$ which is satisfied by a set of assignments $X$ if and only if the corresponding relation satisfies $\phi(R)$). Adding new atoms with these (first order definable) satisfaction conditions, therefore, will increase the expressive power of the logic, in the sense that there will be formulas of this new logic which are not equivalent to any first order formula. 

Does it follow that this new logic will be more expressive also with respect to \emph{sentences}, in the sense that there will exist sentences of this new logic which are not equivalent to any first order sentence? Not necessarily. While this is true for e.g. Dependence Logic and Independence Logic (that is, for the logics obtained by adding the \emph{functional dependence atoms} or the \emph{independence atoms} of Definition \ref{def:depindep} respectively), it is possible to find dependencies for which this is not the case: while adding them to the language of First Order Logic makes it more expressive \emph{with respect to formulas}, every sentence of this new logic is still equivalent to some first order sentence. Dependency atoms, or families of dependency atoms, for which this is the case are called \emph{strongly first order} \cite{galliani2015upwards,galliani2016strongly}.

This asymmetry between expressivity on the level of formulas and expressivity on the level of sentences is one of the most intriguing aspects of Team Semantics; and, in particular, the fact that when using Team Semantics it is possible to generate logics with expressive powers between that of First Order Logic \FO~ and that of Existential Second Order Logic \ESO ~(included) simply by adding to \FO~ atoms with first order definable satisfaction conditions makes Team Semantics an eminently suitable tool for the study and classification of fragments of \ESO, a topic rich with open questions and with deep connections with important complexity-theoretic conjectures. 

The study of strongly first order dependencies, in particular, can be thought of as an attempt to investigate the border between first order and second order ``from below'' by seeking to characterize precisely which choices of dependency atoms (or families of dependency atoms) breach or fail to breach it. The conjecture according to which a dependency is strongly first order if and only if it is definable in terms of upwards closed dependencies and constancy dependencies\footnote{As proved in \cite{galliani2015upwards} these families of dependencies are both strongly first order, as is their union.} is, however, still unproven. In this work, a proof for a special case of it will be found: if a non-trivially-false dependency is strongly first order, is downwards closed, and is furthermore \emph{relativizable} (a new, natural property of dependencies which will be introduced in this work)\footnote{This property is related to the \emph{relativization operator} of \cite{ronnholm2018arity}.} then it is definable in terms of constancy atoms alone. 

A property related to strong first orderness and introduced in the recent article \cite{galliani2018safe} is \emph{safety}. In brief, a dependency (or a set of dependencies) is safe for some logic if it can be added to it without increasing the expressive power (wrt sentences) of the resulting formalism. In particular, a dependency is strongly first order if and only if it is safe for First Order Logic. This notion of safety, thus, generalizes the notion of strong first orderness to logics more expressive than pure First Order Logic, and a complete characterization of safety would provide a full classification of the expressive powers of the logics obtained by adding dependencies to logics with Team Semantics. Safety, however, is a surprisingly delicate notion: for instance, as shown in \cite{galliani2018safe}, constancy -- despite being perhaps the simplest example of strongly first order dependency -- is \emph{not} safe for First Order Logic plus certain classes of dependencies (e.g. unary inclusion dependencies). In this work it will be shown that, in the case of downwards closed dependencies, safety is more robust: any strongly first order dependency is safe for First Order Logic plus any family containing only downwards closed dependencies. This result, aside from being interesting in its own right, will be essential for the characterization of strongly first order downwards closed relativizable dependencies mentioned above. 

Much of the research in the area of Team Semantics thus far has focused on the study of specific logics obtained by adding specific atoms (or specific operators) to First Order Logic with Team Semantics. This is an important direction of research, and many of the resulting logics are of independent interest (and, often, have intriguing connections with the theory of databases). But an equivalently important, if so far understandably\footnote{Indeed, when studying the properties of a new type of semantics, it is a good strategy to begin by identifying and investigating logics that make use of it and are of independent interest.} less studied, one consists in the \emph{classification} of general families of such logics in terms of their relationships and of their meta-logical properties.\footnote{For examples of works along these lines, see for instance \cite{galliani2015upwards} or \cite{kontinen2016decidability}.} The study of strongly first order and safe dependencies, to which this paper contributes, is a promising subtopic of this intriguing and largely unexplored field of research.

\section{Preliminaries}
\subsection{Team Semantics and Dependencies}
As mentioned in the Introduction, Team Semantics generalizes Tarskian semantics for First Order Logic by letting formulas be satisfied or not satisfied by sets of assignments, which are called \emph{teams} for historical reasons: 
\begin{Definition}[Team]
Let $\mathfrak M$ be a first order model (over any signature $\Sigma$) with domain $M$ and let $V$ be a finite set of variables. Then a team $X$ over $\M$ with domain $\dom(X) = V$ is a set of variable assignments $s: V \rightarrow M$ over $\M$.
\end{Definition}
There exists an obvious correspondence between teams and relations: 
\begin{Definition}[Teams to Relations]
Let $X$ be a team over some model $\M$, and let $\tuple t = t_1 \ldots t_n$ be a finite tuple of terms in the signature of $\M$ with variables in $\dom(X)$. Then we write $X(\tuple t)$ for the $n$-ary relation 
\[
	X(\tuple t) = \{\langle t_1^\M(s),t_2^\M(s), \ldots, t_n^\M(s)\rangle : s \in X\}.
\]
	where each $t_i^\M(s)$ is the interpretation of $t_i$ in $\M$ for the assignment $s$. 
\end{Definition} 

Teams can be restricted to a subset of the variables in their domain in the obvious way: 
\begin{Definition}[Team Restriction]
Let $X$ be a team over some model $\M$ and let $V \subseteq \dom(X)$ be a subset of the variables in its domain. Then we write $X_{|V}$ for the restriction of $X$ to the variables in $V$, that is for 
\[
	X_{|V} = \{s_{|V} : s \in X\}
\]
where, for all assignments $s \in X$, $s_{|V}$ is the unique assignment with domain $V$ such that $s_{|V}(v) = s(v)$ for all variables $v \in V$. 
\end{Definition}

It will also be useful, in several places of this work, to consider the subteam of a given team which contains only the assignments satisfying (in the sense of the usual Tarskian Semantics) some first order formula $\theta$. This is defined in the obvious way: 
\begin{Definition}[Team Selection]
Let $X$ be a team over some model $\M$ and let $\theta(\tuple x)$ be a first order formula over the signature of $\M$ whose free variables $\fv(\theta)$ are contained in the domain $\dom(X)$ of $X$. Then
\[
X \upharpoonright \theta(\tuple x) = \{s \in X : \M \models_s \theta(\tuple x)\}
\]
where the expression $\M \models_s \theta(\tuple x)$ means that the assignment $s$ satisfies $\theta(\tuple x)$ in $\M$ according to the usual Tarskian Semantics. 
\end{Definition}

Finally, in order to define Team Semantics we will also need the two following operations: 
\begin{Definition}[(Lax) Supplementation and Duplication]
Let $X$ be a team over some first order model $\M$ and let $H: X \rightarrow \parts(M^k) \backslash \{\emptyset\}$ be a function associating to each assignment $s \in X$ some nonempty set $H(s) \subseteq M^k$ of $k$-tuples of possible values, and let $\tuple v = v_1 \ldots v_k$ be any $k$-tuple of pairwise distinct variable symbols (which may or may not occur already in $\dom(X)$). Then we write $X[H/\tuple v]$ for the team, with domain $\dom(X) \cup \{v_i : i = 1 \ldots k\}$, defined as 
\[
X[H/\tuple v] = \{s[\tuple m/\tuple v] : s \in X, \tuple m \in H(s)\}
\]
where, as usual, $s[\tuple m/\tuple v]$ is the assignment obtained starting from $s$ and fixing the values of the variables $\tuple v = v_1 \ldots v_k$ to $\tuple m = m_1 \ldots m_k$. This team is called the \emph{supplementation} of $X$ along $H$.

For every $\tuple m \in M^k$, we also write $X[\tuple m/\tuple v]$ for the set $\{s[\tuple m / \tuple v] : s \in X\}$, that is for the supplementation of $X$ along $\tuple v$ via some $H$ with $H(s) = \{\tuple m\}$ for all $s \in X$.

Finally, we write $X[M/\tuple v]$ for the largest possible supplementation of $X$, which is also called the \emph{duplication} of $X$ and which extends $X$ by assigning \emph{all} possible values to the variables in $\tuple v$:
\[
X[M/\tuple v] = \{s[\tuple m/\tuple v] : s \in X, \tuple m \in M^k\}
\]
\end{Definition}
\begin{figure}
\[
\begin{array}{l c l}
X = \begin{array}{|c | c|}
\hline
& v_0\\
\hline
s_0 & 0\\
s_1 & 1\\
\hline
\end{array}
&~~~~~~~~~~~~~~~~~~~~~~~~~~~~~~~&
H(s) = \left\{\begin{array}{l l}
\{\langle 1, 0\rangle \} & \text{ if } s = s_0\\
\{\langle 0, 0\rangle, \langle 0, 1\rangle\} & \text{ if } s = s_1
\end{array}
\right.\\
\\
X[H/v_1 v_2] = 
\begin{array}{| c | c | c | c |}
\hline 
& v_0 & v_1 & v_2\\
\hline
s'_0 & 0 & 1 & 0\\
s'_1 & 1 & 0 & 0\\
s'_2 & 1 & 0 & 1\\
\hline
\end{array}
&~~~~~~~~~~~~~~~~~~~~~~~~~~~~~~~&
X[M/v_1 v_2] = 
\begin{array}{| c | c | c | c |}
\hline 
& v_0 & v_1 & v_2\\
\hline
s''_0 & 0 & 0 & 0\\
s''_1 & 0 & 0 & 1\\
s''_2 & 0 & 1 & 0\\
s''_3 & 0 & 1 & 1\\
s''_4 & 1 & 0 & 0\\
s''_5 & 1 & 0 & 1\\
s''_6 & 1 & 1 & 0\\
s''_7 & 1 & 1 & 1\\
\hline
\end{array}
\end{array}
\]
\caption{Supplementation and duplication examples in a model with only two elements $0$ and $1$.}
\end{figure}
The reason why the above supplementation operation is described as ``lax'' is because there also exists in the literature a ``strict'' version, in which $H(s)$ is a singleton for every $s \in X$. As discussed in \cite{galliani12}, the choice between these two operations (as well as between two possible semantics for disjunction) corresponds precisely to the choice between allowing or disallowing nondeterministic strategies in the equivalent imperfect-information game-theoretic semantics; however, when considering non downwards-closed dependencies, the strict variant of Team Semantics can fail to satisfy locality (Proposition \ref{propo:local} in this work) in the sense that the satisfiability of a formula $\phi$ in a team $X$ may depend on the values taken in $X$ by variables which do \emph{not} appear free in $\phi$. Because of this, in this work (as in most of the recent literature in the area of Team Semantics) we will focus only on the lax version of the semantics. 
\begin{Definition}[(Lax) Team Semantics for First Order Logic]
Let $\M$ be a first order model, let $X$ be a team over it, and let $\phi(\tuple x)$ be a first order formula in Negation Normal Form\footnote{It is possible to define Team Semantics for expressions not in Negation Normal Form, as it was done for instance in \cite{vaananen07}. However, doing so requires taking track of positive and negative satisfaction conditions, which makes the semantics more cumbersome; and, as discussed in \cite{kontinenv10}, the usual (``dual'') negation operator has little semantic meaning in certain Team Semantics-based extensions of First Order Logic, as the satisfaction conditions of a formula and of its negation are essentially unrelated. Furthermore, it is often unclear what the interpretation of the dual negation of a dependency atom should be: for instance, in \cite{vaananen07} it was decided that the negations of functional dependence atoms $=\!\!(\tuple x; \tuple y)$ are only satisfied by empty teams, that is they are all equivalent to $\bot$.  For this reason, in this work we will assume that all expressions are in Negation Normal Form. Another possible negation operator, of clearer interpretation in Team Semantics, is the contradictory negation $\M \models_X \sim \phi \Leftrightarrow \M \not \models_X \phi$; but adding it to First Order Logic together with even very simple dependencies (e.g. constancy atoms) brings the expressive power of the resulting formalism all the way up to Second Order Logic. The logic $\FO(\sim)$ obtained by taking First Order Logic (with Team Semantics) and adding to it the contradictory negation (but no dependencies) is however equivalent to First Order Logic wrt sentences, as shown in \cite{galliani2016strongly}, and in \cite{luck2018axioms} an axiomatization for it is found. This operator will not be further discussed in this work.} over the signature of $\M$ such that the free variables of $\phi$ are contained in the domain of $X$. Then we say that $X$ satisfies $\phi$ in $\M$, and write $\M \models_X \phi$, if and only if this follows from the following rules: 
\begin{description}
\item[TS-lit] If $\alpha$ is a first order literal then $\M \models_X \alpha$ if and only if, for all $s \in X$, $\M \models_s \alpha$ according to the usual rules of Tarskian Semantics; 
\item[TS-$\vee$] For all formulas $\psi_1$ and $\psi_2$, $\M \models_X \psi_1 \vee \psi_2$ if and only if there exist teams $Y$ and $Z$ such that $X = Y \cup Z$, $\M \models_{Y} \psi_1$ and $\M \models_{Z} \psi_2$;\footnote{In this rule we do not require that $Y \cap Z = \emptyset$. Doing that would give us the strict semantics for disjunction, which -- as in the case of supplementation and existential quantification --  corresponds to allowing only deterministic strategies in the game theoretic semantics and would result in the failure of Proposition \ref{propo:local} for certain non-downwards closed dependencies.}
\item[TS-$\wedge$] For all formulas $\psi_1$ and $\psi_2$, $\M \models_X \psi_1 \wedge \psi_2$ if and only if $\M \models_X \psi_1$ and $\M \models_X \psi_2$;
\item[TS-$\exists$] For all formulas $\psi$ and all variables $v$, $\M \models_X \exists v \psi$ if and only if there exists some $H:X \rightarrow \parts(M) \backslash \{\emptyset\}$ such that $\M \models_{X[H/v]} \psi$;
\item[TS-$\forall$] For all formulas $\psi$ and all variables $v$, $\M \models_X \forall v \psi$ if and only if $\M \models_{X[M/v]} \psi$. 
\end{description}
Given a first order model $\M$ and a negation normal form, first order sentence $\phi$ over its signature, we say that $\phi$ is true in $\M$ in Team Semantics, and write $\M \models \phi$, if and only if $\M \models_{\{\epsilon\}} \phi$, where $\{\epsilon\}$ is the team containing the unique assignment $\epsilon$ over the empty domain.
\label{def:teamsem}
\end{Definition}
The following two widely known facts describe completely the relationship between Team Semantics and Tarskian Semantics for First Order Logic: 
\begin{Proposition}
Let $\M$ be a first order model, let $X$ be a team over it, and let $\phi(\tuple x) \in \FO$ be a first order formula over the signature of $\M$ and with free variables in $\dom(X)$. Then $\M \models_X \phi(\tuple x)$ if and only if, for all assignments $s \in X$, $\M \models_s \phi$ according to the usual Tarskian Semantics. 
\label{propo:FOform}
\end{Proposition}
\begin{proof}
Straightforward induction. 
\end{proof}

\begin{Proposition}
Let $\M$ be a first order model and let $\phi \in \FO$ be a first order sentence over the signature of $\M$. Then $\M \models \phi$ in the sense of Team Semantics if and only if $\M \models \phi$ in the sense of Tarskian Semantics. 
\label{propo:FOsent}
\end{Proposition}
\begin{proof}
By definition, $\M \models \phi$ in the sense of Team Semantics if and only if $\M \models_{\{\epsilon\}} \phi$. By Proposition \ref{propo:FOform}, this is the case if and only if $\M \models_\epsilon \phi$ in the sense of Tarskian Semantics. This holds if and only if $\phi$ is true in $\M$ according to Tarskian Semantics. 
\end{proof}
Do these two results show that Team Semantics is a pointlessly overwrought, but practically equivalent, variant of Tarskian semantics? Well, no: as mentioned in the Introduction, the richer structure of the satisfaction relation in Team Semantics makes it possible to extend First Order Logic in new ways. In Team Semantics, the satisfaction conditions of a first order formula $\phi(\tuple x)$ always correspond to some first order definable property of relations, in the sense that there exists some first order sentence $\phi'(R)$, where $R$ is a relation symbol not in the signature of $\M$, such that $\M \models_X \phi(\tuple x)$ if and only if $\M[X(\tuple x)/R] \models \phi'(R)$\footnote{Here we write $\M[X(\tuple t)/R]$ for the first order model obtained by adding the relation symbol $R$ to $\M$ -- if it is not present already -- and fixing the relation $X(\tuple t) = \{s(\tuple t) : s \in X\}$ as its interpretation.}: indeed, by Proposition \ref{propo:FOform}, it suffices to take $\phi'(R) := \forall \tuple x(R \tuple x \rightarrow \phi(\tuple x))$. However, a moment's thought shows that there exist first order definable properties of relations that do not correspond to the satisfaction conditions of any first order formula in Team Semantics. For instance,
\begin{Corollary}
There is no first order formula $\phi(x)$ such that $\M \models_X \phi(x)$ if and only if $|X(x)| \leq 1$, that is, if and only if the variable $x$ takes at most one value in $X$.
\label{coro:const}
\end{Corollary}

Thus, it is possible to extend the Team Semantics of First Order Logic by introducing new types of atoms, such as the following \emph{constancy atoms:}
\begin{Definition}[Constancy Atoms, Constancy Logic]
	Constancy Logic\\$\FO(=\!\!(\cdot))$ is the logic obtained by adding to the language of First Order Logic (in Negation Normal Form) \emph{constancy atoms} $=\!\!(x)$ for all variables $x$, with satisfaction conditions 
\begin{description}
\item[TS-const] $\M \models_X =\!\!(x)$ if and only if $|X(x)| \leq 1$. 
\end{description}
\label{def:const}
\end{Definition}
It is then possible to inquire about the properties of this Constancy Logic. For example, by Corollary \ref{coro:const} we know already that there exist formulas in constancy logic which are not equivalent to any first order formula; but what about sentences? Can Constancy Logic be used to define classes of models which are not definable in First Order Logic? As shown in \cite{galliani12}, the answer is no. However, this \emph{is} the case for the following atoms and for the logics that they characterize: 
\begin{Definition}[Functional Dependence, Independence, and Inclusion]
	$\\$Dependence Logic, Independence Logic, and Inclusion Logic are the logics obtained by adding to the language of First Order Logic (in Negation Normal Form) \emph{functional dependence atoms} $=\!\!(\tuple x; \tuple y)$, \emph{independence atoms} $\tuple x \bot_\tuple y \tuple z$, and \emph{inclusion atoms} $\tuple x \subseteq \tuple y$ respectively, with the following semantics: 
\begin{description}
\item[TS-dep] $\M \models_X =\!\!(\tuple x; \tuple y)$ iff for any two $s, s' \in X$, if $s(\tuple x) = s'(\tuple x)$ then \\$s(\tuple y) = s'(\tuple y)$; 
\item[TS-indep] $\M \models_X \tuple x \bot_\tuple y \tuple z$ iff for any two $s, s' \in X$ with $s(\tuple y) = s'(\tuple y)$ there exists some $s'' \in X$ such that $s''(\tuple x \tuple y) = s(\tuple x \tuple y)$ and $s''(\tuple y \tuple z) = s'(\tuple y \tuple z)$; 
\item[TS-inc] $\M \models_X \tuple x \subseteq \tuple y$ iff for any $s \in X$ there exists some $s' \in X$ such that $s(\tuple x) = s'(\tuple y)$. 
\end{description}
\label{def:depindep}
\end{Definition}
As shown in \cite{vaananen07}, \cite{gradel13} and \cite{gallhella13} respectively, Dependence Logic and Independence Logic are equivalent to Existential Second Order Logic \ESO~ over sentences, while Inclusion Logic is equivalent to the positive fragment of Greatest Fixed Point Logic\footnote{This implies, in particular, that Inclusion Logic captures PTIME over finite ordered models.}.  This is the case despite the fact that the satisfaction conditions of the above atoms are easily definable as first order properties of relations; and, as mentioned in the Introduction, the fact that Team Semantics allows to greatly increase the expressive power of First Order Logic via first order definable properties is precisely what makes it an eminently suitable tool for the study of the boundary between first order and higher order logics. Additionally, it is worth remarking here that these dependencies have clear connections with database theory: the relationship between functional dependence and inclusion atoms and the functional \cite{codd72, armstrong74} and inclusion \cite{casanova82} dependencies of database theory is obvious, and as discussed in \cite{engstrom12} there likewise exists a correspondence between independence atoms and database-theoretic \emph{multivalued dependencies} \cite{fagin77}. Additionally, these atoms have a strong doxastic flavour: if a team $X$ represents the set of the states of things that an agent believes possible then $=\!\!(\tuple x, \tuple y)$ can be read as ``if I learned the true value of $\tuple x$, I could infer the value of $\tuple y$'', while $\tuple x \bot_\tuple y \tuple z$ can be read as ``if I learned the true value of $\tuple y$, the value of $\tuple x$ would give me no information whatsoever regarding the value of $\tuple z$'' and $\tuple x \subseteq \tuple y$ can be read as ``every possible value for $\tuple x$ is also a possible value for $\tuple y$''. As discussed in \cite{galliani2015doxastic}, all the connectives and operators of Team Semantics also admit natural interpretations in terms of the dynamics of belief states; but we will not pursue this line of thought any further in this work.

A very fruitful research direction in Team Semantics research so far has consisted in the study of the properties of these logics and of fragments thereof (see e.g. \cite{kontinen2013coherence,durand2012hierarchies,galliani12,gallhella13,galliani13b,kontinen2013axiomatizing, hannula2015axiomatizing, hannula2018hierarchies,durand2016expressivity}). This is a valuable topic of investigation, with rich connections with open problems in the classification of better known fragments of Second Order Logic and in descriptive complexity theory. Furthermore, as briefly mentioned above, these logics are of independent interest, and hence their properties and those of their fragments are deserving of study for their own sake. 

The present work is a contribution to a different -- if obviously related -- research programme, one in which the central topic of investigation is not Dependence Logic or Independence Logic or any other specific Team Semantics-based logic per se but rather \emph{Team Semantics itself}; and, under this perspective, rather than selecting ``interesting'' additional atoms and studying the logics obtained by adding them to First Order Logic we want to select ``interesting'' properties for Team Semantics-based logics and investigate which atoms (or more in general which operators, although as we will see there are plenty of open questions even in the more limited case of atoms) or collections of atoms satisfy them if added to the language of First Order Logic. 

In order to do this properly, we must first clarify exactly what a dependency atom \emph{is} in Team Semantics. The following definition, from \cite{kuusisto2015}, is a natural starting point: 
\begin{Definition}[Dependency]
For any $k \in \mathbb N$, a $k$-ary \emph{dependency} $\D$ is a class of models over the signature $\{R\}$ which is closed under isomorphisms, where $R$ is a $k$-ary relation symbol. 

	Given a family of dependencies $\DD = \{\D_i : i \in I\}$, we write $\FO(\DD)$ for the logic obtained by adding to the language of First Order Logic (in Negation Normal Form) all dependency atoms $\D \tuple t$, where $\D \in \DD$ and $\tuple t$ is a tuple of terms of length equal to the arity of $\D$. The Team Semantics for $\FO(\DD)$ is defined precisely as in Definition \ref{def:teamsem}, with the additional condition 
\begin{description}
\item[TS-$\D$] For all models $\M$ with domain $M$, all teams $X$ over $\M$, all $\D \in \DD$ and all tuples $\tuple t$ of terms in the signature of $\M$ with variables in the domain of $X$ and of length equal to the arity of $\D$, $\M \models_X \D \tuple t$ if and only if $\langle M, X(\tuple t)\rangle \in \D$.\footnote{In this work we identify the tuple $\langle M, R_1 \ldots R_n\rangle$ with the model $\M$ with domain $M$ and relations $R_1 \ldots R_n$. When no ambiguity regarding the choice of model is possible, as in this case, we also use the same symbol $R_i$ for the relation symbol and for its interpretation $R^\M_i$.}
\end{description}
\label{def:generaldeps}
\end{Definition}

The following result, that can be found in \cite{kontinen2016decidability} together with a number of results concerning the complexity of the satisfiability problem for fragments of logics with generalized dependencies, shows that -- regardless of the choice of $\DD$ -- $\FO(\DD)$ is \emph{local} in the sense that the values of variables which do not appear free in a formula are irrelevant to its satisfaction or lack thereof: 
\begin{Proposition}[Locality]
Let $\DD$ be any family of dependencies, let $\phi$ be a formula of $\FO(\DD)$, and let $\fv(\phi)$ be the set of its free variables. Then, for all models $\M$ whose signature contains that of $\phi$ and for all teams $X$ with domain containing $\fv(\phi)$, 
\[
	\M \models_X \phi \Leftrightarrow \M \models_{X_{|\fv(\phi)}} \phi. 
\]
\label{propo:local}
\end{Proposition}
\begin{proof}
By structural induction.
\end{proof}

It is easy to verify that all the dependency atoms discussed so far are special cases of Definition \ref{def:generaldeps}. Note, furthermore, that nothing in Definition \ref{def:generaldeps} requires the class $\D$ to be first order definable. For instance, $\U = \{\langle M, R\rangle : |R| \text{ is uncountable}\}$\footnote{Strictly speaking, this is not a set but a proper class since $M$ ranges over all possible (uncountable) sets of elements. This is not a problem for the purposes of this work -- see Definition \ref{def:generaldeps} --  and we will use expressions such as the above for defining the interpretations of dependencies.}  is a perfectly acceptable unary dependency, and the corresponding satisfaction condition in Team Semantics is: $\M \models_X \U t$ if and only if $X(t)$ is uncountable. 

If, as in the case of $\U$, the satisfaction conditions of a dependency atom $\D$ are not first order definable then obviously not all sentences of $\FO(\D)$ are equivalent to first order sentences: indeed, it is easy to see that the $\FO(\D)$ sentence $\forall \tuple x (\lnot R \tuple x \vee (R \tuple x \wedge \D \tuple x))$ characterizes precisely the class $\D$, which is not first order definable by assumption. A more interesting case is the one in which $\D$ -- understood as a class of models -- \emph{is} first order definable: 

\begin{Definition}[First Order Dependencies]
A $k$-ary dependency $\D$ is \emph{first order} if there exists some sentence $\phi(R)$ in the signature $\{R\}$, where $R$ is $k$-ary, such that $\D = \{\langle M, R\rangle : \langle M, R\rangle \models \phi(R)\}$.
\label{def:fodep}
\end{Definition}
\begin{Proposition}
Constancy atoms, functional dependence atoms, independence atoms and inclusion atoms are all first order. 
\end{Proposition}
\begin{proof}
Choose for $\phi(R)$ the expressions $\forall \tuple x \tuple x' (R \tuple x \wedge R \tuple x' \rightarrow \tuple x = \tuple x')$, \\$\forall \tuple x \tuple x' \tuple y (R \tuple x \tuple y \wedge R \tuple x' \tuple y \rightarrow \tuple x = \tuple x')$, $\forall \tuple x \tuple x' \tuple y \tuple z \tuple z' (R \tuple x \tuple y \tuple z \wedge R \tuple x' \tuple y \tuple z' \rightarrow R \tuple x \tuple y \tuple z')$ and \\$\forall \tuple x \tuple y (R \tuple x \tuple y \rightarrow \exists \tuple z R \tuple z \tuple x)$ respectively.
\end{proof}

The following properties, well studied in the literature, provide a useful starting point for the classification of dependencies: 
\begin{Definition}[Closure Properties for Dependencies]
Let $\D$ be any dependency. Then 
\begin{itemize}
\item $\D$ has the \emph{empty team property} if, for all sets of elements $M$, $\langle M, \emptyset\rangle \in \D$; 
\item $\D$ is \emph{downwards closed} if for all relations $R$ over a domain $M$, $\langle M, R\rangle \in \D$ and $Q \subseteq R$ imply $\langle M, Q\rangle \in \D$; 
\item $\D$ is \emph{upwards closed} if for all relations $R$ over some domain $M$, $\langle M, R\rangle \in \D$ and $Q \supseteq R$ imply $\langle M, Q\rangle \in \D$; 
\item $\D$ is \emph{union closed} if for all families $\{R_i : i \in I\}$ of relations of the same arity over the same domain $M$ such that $\langle M, R_i\rangle \in \D$ for all $i \in I$ we have that $\left \langle M, \bigcup_i R_i\right \rangle \in \D$ as well.

\end{itemize}
\end{Definition}
Aside from upwards closure, these properties are preserved by Team Semantics in the following sense: 
\begin{Proposition}
Let $\DD$ be a family of dependencies and let $\phi(\tuple x)$ be any formula of $\FO(\DD)$. Then, for all models $\M$ with signature containing the signature of $\phi(\tuple x)$, 
\begin{enumerate}
\item If all $\D \in \DD$ have the empty team property then $\M \models_\emptyset \phi$; 
\item If all $\D \in \DD$ are downwards closed and $X$ is a team such that $\M \models_X \phi$ then $\M \models_Y \phi$ for all $Y \subseteq X$;
\item If all $\D \in \DD$ are union closed and $(X_i : i \in I)$ is a family of teams such that $\M \models_{X_i} \phi$ for all $i \in I$ then $\M \models_{\bigcup_i X_i} \phi$.
\end{enumerate}
\label{propo:clos_cons}
\end{Proposition}
\begin{proof}
Straightforward induction. 
\end{proof}
Closure properties such as these are useful to establish nondefinability relations between dependencies. 
\begin{Definition}[Definability of Dependencies]
Let $\DD$ be a family of dependencies and let $\E$ be another dependency. Then $\E$ is said to be \emph{definable} in $\FO(\DD)$ if for all tuples $\tuple t$ of terms there exists some formula $\phi(\tuple t)$ such that $\M \models_X \E \tuple t$ if and only if $\M \models_X \phi(\tuple t)$.\footnote{The choice of $\tuple t$ may force us to rename bound variables in $\phi$; but aside from that, it is not difficult to see that if such a $\phi$ exists then it works for any $\tuple t$.}
\label{def:def}
\end{Definition}
It is easy to verify that the dependencies discussed thus far all have the empty team property; that none of them is upwards closed; that constancy and functional dependency are downwards closed but not union closed; that inclusion is union closed but not downwards closed; and that independence is neither union closed nor downwards closed. By Proposition \ref{propo:clos_cons}, this allows us to infer at once for example that neither inclusion nor functional dependence are definable in terms of the other, and that neither of them alone suffices to define independence.\footnote{However, as shown in \cite{galliani12}, independence atoms are definable if we have \emph{both} inclusion atoms and functional dependence atoms.}

It follows from known results in the literature that functional dependence atoms, inclusion atoms and independence atoms are ``maximal'' among the corresponding classes of first order dependencies. More precisely, for all first order dependencies $\D$,
\begin{enumerate}
\item If $\D$ has the empty team property and is downwards closed then it is definable in terms of functional dependence atoms (\cite{kontinenv09}); 
\item If $\D$ has the empty team property then it is definable in terms of independence atoms (\cite{galliani12}); 
\item If $\D$ has the empty team property and is union closed then it is definable in terms of inclusion atoms (\cite{gallhella13}). 
\end{enumerate} 
In fact, the results concerning functional dependence atoms and independence atoms are stronger, in that they also hold for \emph{existential second order} (rather than first order) dependencies $\D$.\footnote{Analogously to Definition \ref{def:fodep}, a dependency $\D$ is existential second order if there exists some $\ESO$~ sentence $\phi(R)$ such that $\D = \{\langle M, R\rangle : \langle M, R\rangle \models \phi(R)\}$.} This is not, however, the case for Inclusion Logic: while all first order union-closed dependencies with the empty team property are definable in it, as discussed in \cite{gallhella13} it is possible to find existential second order union-closed dependencies with the empty team property which are not definable in Inclusion Logic.

We conclude this section by providing a few definitions and elementary results that will be used in this work. 

\begin{Definition}[$\exists \tuple v$ and $\forall \tuple v$]
Let $\tuple v = v_1 \ldots v_k$ be a tuple of pairwise distinct variables. Then, for all formulas $\phi$, we write $\exists \tuple v \phi$ as a shorthand for $\exists v_1 \exists v_2 \ldots \exists v_k \phi$, and we write $\forall \tuple v \phi$ as a shorthand for $\forall v_1 \forall v_2 \ldots \forall v_k \phi$. 
\end{Definition}
\begin{Proposition}
For all models $\M$, all families of dependencies $\DD$, all formulas $\phi \in \FO(\DD)$, all tuples $\tuple v = v_1 \ldots v_k$ of variables, all models $\M$ whose domain contains the signature of $\phi$ and all teams $X$ whose domain contains $\fv(\phi) \backslash \{v_1 \ldots v_k\}$ 
\[
\M \models_X \exists \tuple v \phi \Leftrightarrow \text{there is a } H: X \rightarrow \parts(M^k) \backslash \{\emptyset\} \text{ s.t. } \M \models_{X[H/\tuple v]} \phi
\]
and 
\[
\M \models_X \forall \tuple v \phi \Leftrightarrow \M \models_{X[M/\tuple v]} \phi.
\]
\end{Proposition}
\begin{proof}
Trivial induction on the length $k$ of the tuple $\tuple v$.
\end{proof}
\begin{Definition}[Selective Implication]
For all families of dependencies $\DD$, all formulas $\phi \in \FO(\DD)$ and all first order formulas $\theta \in \FO$, we write $\theta \hookrightarrow \phi$ as a shorthand for $(\lnot \theta) \vee (\theta \wedge \phi)$ where $\lnot \theta$ is the first order negation normal form expression equivalent to the negation of $\theta$. 
\label{def:selimp}
\end{Definition}
\begin{Proposition}
For all models $\M$, all teams $X$, all families of dependencies $\DD$, all formulas $\phi \in \FO(\DD)$ and all formulas $\theta \in \FO$, 
\[
\M \models_X \theta \hookrightarrow \phi \Leftrightarrow \M \models_{X \upharpoonright \theta} \phi.
\]
\label{propo:selimp}
\end{Proposition}
\begin{proof}
Follows easily from Proposition \ref{propo:FOform} and from the rule for disjunction in Team Semantics. 
\end{proof}
\begin{Definition}[Boolean Disjunction]
Let $\phi, \psi \in \FO(\DD)$ for some choice of $\DD$. Then we write $\phi \sqcup \psi$ for the $\FO(\DD \cup \{ =\!\!(\cdot)\})$-formula 
\[
\exists z_1 z_2 (=\!\!(z_1) \wedge =\!\!(z_2) \wedge (( z_1 = z_2 \wedge \phi) \vee (z_1 \not = z_2 \wedge \psi)))
\]
where $z_1$ and $z_2$ are two new variables not occurring in $\phi$ or $\psi$ and $=\!\!(z_i)$ is the constancy atom of Definition \ref{def:const}, which holds in a team $X$ if and only if $z_i$ takes at most one value in it. 
\label{def:booldisj}
\end{Definition}
\begin{Proposition}
For all models $\M$, all families of dependencies with the empty team property $\DD$, all formulas $\phi, \psi \in \FO(\DD)$ and all teams $X$, 
\[
\M \models_X \phi \sqcup \psi \Leftrightarrow \M \models_X \phi \text{ or } \M \models_X \psi. 
\]
\label{propo:booldisj}
\end{Proposition}
\begin{proof}
Straightforward from definitions (observe, however, that the requirement that all $\D \in \DD$ have the empty team property cannot be removed). 
\end{proof}

\begin{Proposition}[Positive Occurrences of Relations]
Let $\DD$ be any family of dependencies and let $\phi$ be any formula of $\FO(\DD)$ in which some relation symbol $R$ occurs only positively.\footnote{Since all our expressions are in Negation Normal Form, this is the same as saying that no negated literal $\lnot R \tuple t$ appears in $\phi$.} Then $\phi$ is \emph{upwards closed} in $R$, in the sense that 
\[
\M \models_X \phi, R^\M \subseteq Q \Rightarrow \M[Q/R] \models_X \phi
\]
for all suitable models $\M$, teams $X$ and relations $Q$. 
\label{propo:posTS}
\end{Proposition}
\begin{proof}
Trivial by structural induction.
\end{proof}


\subsection{Strongly First Order and Safe Dependencies}
As already mentioned, in \cite{galliani12} it was shown that the constancy atoms of Definition \ref{def:const} do not increase the expressive power of First Order Logic with respect to sentences, in the sense that every sentence of the logic $\FO(=\!\!(\cdot))$ obtained by adding these atoms to First Order Logic is equivalent to some first order sentence. In other words, First Order Logic is not a natural ``stopping point'' in the family of the logics based on Team Semantics, as it is possible to find expressions (e.g. constancy atoms) that as per Corollary \ref{coro:const} cannot be defined in terms of it but however do not add to the expressive power of its sentences. 

Thus, we say that constancy atoms are \emph{strongly first order} according to the following definition:
\begin{Definition}[Strongly First Order Dependencies and Families]
Let $\DD$ be a family of first order dependencies. Then $\DD$ is \emph{strongly first order} if every sentence of $\FO(\DD)$ is equivalent to some first order sentence. 

A single dependency $\D$ is said to be strongly first order if the singleton $\{\D\}$ is strongly first order in the above sense. 
\label{def:stronglyFO}
\end{Definition}

Constancy atoms are not the only strongly first order dependencies, and the logic $\FO(=\!\!(\cdot))$ is not a natural ``stopping point'' in the above sense\footnote{This notion will be made more precise in a moment.} either. This follows at once from the following result from \cite{galliani2015upwards}: 
\begin{Theorem}
Let $\DD^\uparrow$ be the family of all first order upwards closed dependencies. Then the family $\DD^\uparrow \cup \{=\!\!(\cdot)\}$, which contains all those dependencies and also constancy atoms, is strongly first order. 
\label{thm:upwards}
\end{Theorem}
Using this result, it is not difficult to find additional strongly first order dependencies. For instance, as mentioned in \cite{galliani2015upwards}, the contradictory negations of inclusion dependencies 
\[
	\M \models_X \tuple x \not \subseteq \tuple y \Leftrightarrow \exists s \in X \forall s' \in X s(\tuple x) \not = s'(\tuple y)
\]
are all strongly first order, since they are definable (in the sense of Definition \ref{def:def}) in terms of constancy atoms and upwards closed dependencies. To the knowledge of the author all strongly first order dependencies known so far are also definable in terms of $\FO(\DD^\uparrow, =\!\!(\cdot))$.\footnote{We commit a minor abuse of notation and write  $\FO(\DD^\uparrow, =\!\!(\cdot))$ rather than  $\FO(\DD^\uparrow \cup \{=\!\!(\cdot)\})$.} This justifies the following\\

\noindent\textbf{Conjecture 1:} A dependency $\D$ is strongly first order if and only if it is definable in $\FO(\DD^\uparrow, =\!\!(\cdot))$.\\

As already mentioned, the main result of this work will be the proof of a special case of this conjecture. \\

In \cite{galliani2018safe}, the notion of strong first orderness was generalized to the following notion of \emph{safety}:
\begin{Definition}[Safe Dependencies and Families]
Let $\LL$ be a logic based on Team Semantics\footnote{For the purposes of this work, we can always assume that $\LL$ is of the form $\FO(\DD)$ for some choice of $\DD$; but nothing prevents applying this definition to another logic -- for instance, to some fragment of First Order Logic or to some extension of it by additional connectives and operators.} and let $\EE$ be a family of dependencies. Then $\EE$ is \emph{safe} for $\LL$ if every sentence of $\LL(\EE)$ is equivalent to some sentence of $\LL$.  

A single dependency $\E$ is said to be safe for $\LL$ if the singleton $\{\E\}$ is safe for $\LL$ in the above sense. 

A dependency $\E$ or a dependency family $\EE$ is said to be safe for a dependency $\D$ or for a dependency family $\DD$ if it is safe for the logic $\FO(\D)$ (resp. $\FO(\DD)$). 
\end{Definition}
By definition, a dependency (or a family of dependencies) is strongly first order if and only if it is safe for $\FO$.  The notion of safety can be used to make precise the informal concept of ``natural stopping point'' mentioned above:
\begin{Definition}[Definitionally Closed Logic]
A logic $\LL$, based on Team Semantics, is \emph{definitionally closed} if and only if the following two properties are equivalent for all dependencies $\D$: 
\begin{enumerate}
\item $\D$ is safe for $\LL$; 
\item $\D$ is definable in $\LL$. 
\end{enumerate}
\end{Definition}

Thus, Corollary \ref{coro:const} and the fact that constancy is strongly first order imply at once that First Order Logic is not definitionally closed, while \textbf{Conjecture 1} would have as a direct consequence that $\FO(\DD^\uparrow, =\!\!(\cdot))$ is definitionally closed (and is, in fact, the only ``definitional closure'' of $\FO$). More in general, a full characterization of the definitionally closed extensions of First Order Logic would be a very valuable contribution to the classification of logics based on Team Semantics. 

Safety, however, is a delicate property: in particular, as shown in \cite{galliani2018safe}, constancy dependencies are not safe for $\FO(\subseteq_1)$, where $\subseteq_1$ is the family of only \emph{unary} inclusion dependencies\footnote{More precisely, $\FO(\subseteq_1)$ adds to First Order Logic all inclusion atoms $v_1 \subseteq v_2$, where $v_1$ and $v_2$ are single variables (not tuples of variables).}, despite being as already mentioned strongly first order (that is, safe for $\FO$). Thus, a dependency $\D$ may be safe for a logic $\LL$ but unsafe for some other logic $\LL'$ which strictly contains $\LL$. This is reminiscent of certain phenomena in the theory of second-order generalized quantifiers \cite{kontinen2010definability},\footnote{The author thanks an anonymous reviewer for pointing out this connection.} and it suggests that the problem of fully characterizing safety relations and definitionally closed logics will not be of easy solution. 

Nonetheless, studying the properties of safety is a fruitful endeavour. In particular, in this work we will prove that strongly first order dependencies are safe for $\FO(\DD)$ whenever $\DD$ contains only downwards closed dependencies, and use this result to prove -- as already mentioned -- a special case of \textbf{Conjecture 1}. 
\subsection{Relativizable Dependencies}
\label{sec:relativizable}
Sometimes it may be necessary to check whether a dependency holds not with respect to the current model, but with respect to some submodel thereof. This justifies the following definitions:\footnote{This notion of relativization of dependencies is closely related to the notion of relativization of formulas in Inclusion-Exclusion Logic discussed by R\"onnholm in \cite{ronnholm2018arity}.}	
\begin{Definition}[Relativized Dependencies]
	Let $P$ be any unary relation symbol, let $\D$ be any $k$-ary dependency, and let $\tuple t$ be a tuple of terms of length $k$. Then for all first order models $\M$ and all teams $X$ over $\M$ with the variables that appear in $\tuple t$ in their domain, $\M \models_X \D^P \tuple t$ if and only if 
	\begin{enumerate}
		\item For all $t \in \tuple t$, $X(t) \subseteq P^\M$;\footnote{This condition is actually a consequence of the next one, but we state it explicitly for clarity.} 
		\item $\langle P^\M, X(\tuple t)\rangle \in \D$. 
	\end{enumerate}
\end{Definition}
\begin{Definition}[Relativizable Dependencies]
	A dependency $\D$ is relativizable if every sentence of $\FO(\D^P)$ is equivalent to some sentence of $\FO(\D)$ over the same signature. 
\end{Definition}

It follows from the above definitions that if $\D$ is relativizable and strongly first order then every sentence constructed out of relativized dependency atoms $\D^P \tuple t$ and first order connectives and literals is equivalent to some first order sentence.

All the dependency atoms widely studied in the literature -- functional dependence, inclusion, independence and so on -- are relativizable. This can be verified by observing that they have the following, stronger property: 
\begin{Definition}[Closed-World Dependencies]
	A dependency $\D$ is \emph{closed} \\\emph{world} if, for all sets of elements $M$ and all relations $R$, if $M' = \{m \in M: \exists \tuple a \in R \text{ s.t. } m \in \tuple a\}$ is the set of all elements occurring in any tuple of $R$ then 
	\[
		\langle M, R\rangle \in \D \Leftrightarrow \langle M', R \rangle \in \D.
	\]
\end{Definition}
This notion is related to the \emph{closed-world assumption} employed in database theory and knowledge representation, and it is also similar (albeit not identical) to the notion of \emph{dependencies closed under substructures} of \cite{kontinen2016decidability}: in short, if $\D$ is closed-world the validity of some atom $\D \tuple v$ in a team $X$ cannot be affected by the existence (or non-existence) in the model of elements that do not appear as possible values of $\tuple v$ in $X$.

It is easy to see that the dependencies examined so far are all closed-world. A first order, unary dependency notion that is not closed-world is \emph{non-totality} $\NT = \{\langle M, P\rangle : P \not = M\}$,
 corresponding to the atoms $\M \models_X \NT (t) \Leftrightarrow X(t) \not = M$.
Indeed, if $\langle M, P\rangle \in \NT$ and we restrict $M$ to the set $M'$ of the elements which occur in $P$ then $(M', P) \not \in \NT$. However, non-totality is definable in terms of constancy, which \emph{is} closed-world: 
\[
	\M \models_X \NT(t) \Leftrightarrow \M \models_X \exists z (=\!\!(z) \wedge z \not = t).
\]
Furthermore, despite not being closed-world, non-totality is still relativizable: indeed, it is straightforward to check that $\NT^P(t)$ is logically equivalent to $P t \wedge \forall x ( (x = t \vee \lnot P x) \hookrightarrow \NT(x))$.\footnote{See Definition \ref{def:selimp} and Proposition \ref{propo:selimp} for the interpretation of the selective implication $\hookrightarrow$. The intuition here is that we force $x$ to take all values outside $P$ as well as all possible values of $t$: so if $x$ does not take all possible values, there must be some value \emph{inside $P$} that $t$ does not take.}

\begin{Proposition}
	If a dependency $\D$ is closed world then it is relativizable.
\end{Proposition}
\begin{proof}
	If $\D$ is closed world and all elements that appear in $X(\tuple v)$ are in $P$, 
	\begin{align*}
		\M \models_X \D^P \tuple v & 
		\Leftrightarrow \langle P, X(\tuple v)\rangle \in \D\\
		& \Leftrightarrow \langle M', X(\tuple v)\rangle \in \D\\
		& \Leftrightarrow \langle M, X(\tuple v)\rangle \in \D
		\Leftrightarrow \M \models_X \D \tuple v 
	\end{align*}
	where $M' = \{s(v_i) : s \in X, v_i \in \tuple v\}$ is the set of all values that variables in $\tuple v$ take in $X$. Thus, $\D^P \tuple v$ is definable in $\FO(\D)$ simply as $\left(\bigwedge_{v_i \in \tuple v} P v_i\right) \wedge \D \tuple v$, and hence $\D$ is relativizable. 
\end{proof}

As we saw, it is easy to find dependencies which are not closed-world, like non-totality. However, it is less obvious to find examples of dependencies that are not relativizable, which led the author to pose the following\\

\noindent \textbf{Problem (Solved by Fausto Barbero):} Find a dependency $\D$ that is not relativizable, or prove that all dependencies are relativizable.\\

Fausto Barbero, in a personal communication, pointed out that a counting argument shows that there exist dependencies (not necessarily first order) that are not relativizable. A concrete example is given by the unary dependency $\mathbf{I}_\infty = \{\langle M, P\rangle: M \text{ is infinite}\}$, for which $\M \models_X \mathbf{I}_\infty(v)$ iff the model $\M$ is infinite regardless of $X$. Indeed, the $\FO(\mathbf{I}_\infty^{P})$ sentence $\exists v \mathbf{I}_\infty^P(v)$ is true if and only if $P$ is infinite. However, there is no $\FO(\mathbf{I}_\infty)$ sentence over the signature $\{P\}$ that satisfies this property, because in an infinite model every occurrence $\mathbf{I}_\infty(v)$ of $\mathbf{I}_\infty$ is equivalent to $\top$ and a standard back-and-forth argument shows that there is no first order sentence which is true in an infinite model if and only if $P$ is infinite. The following variant of the above problem is, to the knowledge of the author, still open, and if solved it would suffice to remove the requirement of relativizability from the results of this paper:\\

\noindent \textbf{Open Problem:} Find a strongly first order dependency $\D$ which is not relativizable, or prove that all such dependencies are relativizable. 
\section{Strongly first order dependencies are safe for downwards closed dependencies}
As already mentioned, constancy dependencies are strongly first order -- that is, safe for $\FO$ -- but they are not safe for unary inclusion dependencies. However, in this section, we will see that strongly first order dependencies are safe for \emph{downwards closed} dependencies. Let us begin by recalling an easy variant of Lyndon's Theorem: 
\begin{Proposition}
	Let $\theta(S_1 \ldots S_n)$ be a first order sentence in a signature containing the symbols $S_1 \ldots S_n$ (and possibly others) which is \emph{upwards closed} in all $S_i$, in the sense that, for any model $\M$, if $\M \models \theta(S_1 \ldots S_n)$ and $S'_i \supseteq S_i$ for all $i$ then $\M \models \theta(S'_1 \ldots S'_n)$ Then $\theta(S_1 \ldots S_n)$ is equivalent to some first order $\theta'(S_1 \ldots S_n)$ in which all $S_i$ occur only positively. 
	\label{propo:lyndon}
\end{Proposition}
\begin{proof}
	Recall (see e.g. Theorem 8.3.3 of \cite{hodges97b}) that if $\theta$ is not equivalent to any formula in which the $S_i$ occur only positively\footnote{But in which the identity symbol may occur positively or negatively - this is necessary to guarantee bijectivity.} then it is not preserved under bijective homomorphisms that fix all relations other than $S_1 \ldots S_n$: therefore, there exist two models $\mathfrak A$ and $\mathfrak B$ such that $\mathfrak A  \models \phi$, $\mathfrak B \not \models \phi$, and there is a bijective homomorphism $\mathfrak h: \mathfrak A \rightarrow \mathfrak B$ that fixes all relations other than $S_1 \ldots S_n$. For all $i = 1 \ldots n$, let $T_i = \mathfrak h[S_i^{\mathfrak A}]$ be the image of $S_i^{\mathfrak A}$ under $\mathfrak h$. Then since $\mathfrak h$ is a homomorphism, $T_i \subseteq S_i^{\mathfrak B}$; since $\mathfrak h$ is bijective and fixes all other relations, $\mathfrak B[T_1 \ldots T_n/S_1 \ldots S_n]$ is isomorphic to $\mathfrak A$, and hence $\mathfrak B \models \phi(T_1 \ldots T_n)$; and thus, since by hypothesis $\mathfrak B \not \models \phi(S_1 \ldots S_n)$, $\phi$ is not upwards closed in $S_1 \ldots S_n$. 
\end{proof}

The main idea of this section will be to ``extract'' the downwards closed dependency atoms $\D \tuple t$ ($\D \in \DD$) from a sentence $\phi \in \FO(\DD, \SSd)$, leaving only a sentence of $\FO(\SSd)$; then use the fact that $\SSd$ is strongly first order to find an equivalent first order sentence; and finally add back the dependencies previously removed. 

The following lemmas show how to do the first step of this procedure: 
\begin{Lemma}
	Let $\phi$ be a $\FO(\DD)$ formula, where $\DD$ is a set of dependencies $\{\D_0, \D_1, \ldots\}$ such that $\D_0$ is downwards closed. Suppose furthermore that $S$ is a new relation symbol with the same arity $k$ as $\D_0$, and let $\phi^*(S)$ be the formula obtained by replacing a \emph{single} instance $\D_0 \tuple t$ of $\D_0$ with $S \tuple t$. Then for all models $\M$ with signature containing the signature of $\phi$ and for all teams $X$ with domain containing the free variables of $\phi$, $\M \models_X \phi$ if and only if there exists a relation $S \subseteq M^{k}$ such that $\langle M, S\rangle \in \D_0$ and  $\M \models_X \phi^*(S)$.
\end{Lemma}
\begin{proof}
	By structural induction on $\phi$: 
	\begin{itemize}
		\item If $\phi$ is of the form $\D_0 \tuple t$ then $\phi^*(S)$ is simply $S \tuple t$. Suppose that $\M \models_X \D_0 \tuple t$; then for $S = X(\tuple t)$ we have that $\langle M, S\rangle \in \D_0$ and $\M \models_X S \tuple t$, as required. Conversely, suppose that such a $S$ exists. Then since $\M \models_X S \tuple t$ we have that $X(\tuple t) \subseteq S$; but then, since $\D_0$ is downwards closed and $\langle M, S\rangle \in \D_0$ we have that $\langle M, X(\tuple t)\rangle \in \D_0$ as well, and hence $\M \models_X \D_0 \tuple t$ as required. 

		\item If $\phi$ is of the form $\psi_1 \vee \psi_2$ then -- assuming without loss of generality that the instance of $\D_0$ which we are replacing is in $\psi_1$ -- $\phi^*(S)$ is $\psi_1^*(S) \vee \psi_2$. Then $\M \models_X \psi_1 \vee \psi_2$ iff $X = Y \cup Z$ for two $Y, Z$ such that $\M \models_Y \psi_1$ and $\M \models_Z \psi_2$ iff (by induction hypothesis) there  exist a relation $S$  and two $Y$, $Z$ such that $\langle M, S\rangle \in \D_0$, $X = Y \cup Z$, $\M \models_Y \psi_1^*(S)$ and $\M \models_Z \psi_2$ iff there exists a relation $S$ such that $\langle M, S\rangle \in \D_0$ and $\M \models_X \phi^*(S)$, as required. 
		\item If $\phi$ is of the form $\psi_1 \wedge \psi_2$ then -- assuming without loss of generality that the instance of $\D_0$ which we are replacing is in $\psi_1$ -- $\phi^*(S)$ is $\psi_1^*(S) \wedge \psi_2$. Then $\M \models_X \phi$ iff $\M \models_X \psi_1$ and $\M \models_X \psi_2$ iff (by induction hypothesis) there is some $S$ such that $\langle M,S\rangle \in \D_0$, $\M \models_X \psi_1^*(S)$ and $\M \models_X \psi_2$ iff there is some $S$ such that $\langle M, S \rangle \in \D_0$ and $\M \models_X \phi^*(S)$. 
		\item If $\phi$ is of the form $\exists v \psi$ then $\phi^*(S)$ is $\exists v \psi^*(S)$. Then $\M \models_X \phi$ iff there exists some $H$ such that $\M \models_{X[H/v]} \psi$ iff there exist some $S$ and $H$ such that $\langle M,  S\rangle \in \D_0$ and $\M \models_{X[H/v]} \psi^*(S)$ iff there exists some $S$ such that $\langle M, S\rangle \in \D_0$ and $\M \models \phi^*(S)$. 
		\item If $\phi$ is of the form $\forall v \psi$ then $\phi^*(S)$ is $\forall v \psi^*(S)$. Then $\M \models_X \phi$ iff $\M \models_{X[M/v]} \psi$ iff there exists some $S$ such that $\langle M, S\rangle \in \D_0$ and $\M \models_{X[M/v]} \psi^*(S)$ iff there exists some $S$ such that $\langle M,S\rangle \in \D_0$ and $\M \models \phi^*(S)$. 
	\end{itemize}
\end{proof}
\begin{Lemma}
	Let $\SSd$ be a strongly first order family of dependencies, let $\DD$ be a set of downwards closed dependencies, and let $\phi$ be a sentence of $\FO(\SSd, \DD)$. Then there exists a first order sentence $\chi(S_1 \ldots S_n)$, whose signature contains the signature of $\phi$ as well as new relation symbols $S_1 \ldots S_n$, and dependencies $\D_{j(1)} \ldots \D_{j(n)} \in \DD$ such that, for all models $\M$, $\M \models \phi$ if and only if there exist relations $S_1 \ldots S_n$ such that 
	\begin{itemize}
		\item For all $i = 1 \ldots n$, $\langle M, S_i\rangle \in \D_{j(i)}$; 
		\item $\M \models \chi(S_1 \ldots S_n)$. 
	\end{itemize}
	
	Moreover, all $S_i$ occur only positively in $\chi$.
\label{lemma:removeSFO}
\end{Lemma}
\begin{proof}
	Starting from $\phi$, we can apply iteratively the previous lemma to remove all instances of dependencies $\D \in \DD$ from $\phi$. In this way, we obtain a sentence $\theta(S_1 \ldots S_n)$ in $\FO(\SSd)$ which -- aside from being in $\FO(\SSd)$ rather than in $\FO$ -- would satisfy our requirements.\footnote{As per Definition \ref{def:generaldeps}, in $\FO(\DD)$, the atoms $\D \tuple t$ for $\D \in \DD$ cannot occur negated. Indeed, $\FO(\DD)$ contains only expressions in Negation Normal Form with first order literals, first order connectives and atoms $\D \tuple t$.} But since $\SSd$ is a strongly first order family of dependencies, $\theta(S_1 \ldots S_n)$ is logically equivalent to some first order $\chi(S_1 \ldots S_n)$. Furthermore, since the $S_i$ occur only positively in $\theta$ by Proposition \ref{propo:posTS} it is the case that $\theta$ is upwards closed in all $S_i$, and hence so is $\chi$, and hence -- by Proposition \ref{propo:lyndon} -- we can also require that all $S_i$ occur only positively in $\chi$. 
\end{proof}
Note that there is no guarantee that the first order sentence $\chi(S_1 \ldots S_n)$ obtained via Lemma \ref{lemma:removeSFO} would contain only one occurrence of each symbol $S_i$, as the translation from $\FO(\SSd)$ to $\FO$ may introduce additional occurrences. However, as we will now see, as long as the $S_i$ occur only positively this can be dealt with: 
\begin{Lemma}
Let $\D$ be a downwards closed dependency, let $\chi(S)$ be a first order formula where $S$ is a relational symbol that occurs only positively in $\chi$, and let $\chi'(W_1 \ldots W_n)$ be the formula obtained by replacing each occurrence of $S$ in $\chi$ with a different new symbol $W_i$. Then the following are equivalent for all suitable models $\M$ and assignments $s$: 
\begin{enumerate}
\item There exists some $S$ such that $\langle M, S\rangle \in \D$ and $\M \models \chi(S)$; 
\item There exist $W_1 \ldots W_n$ such that $\left\langle M,  \left(\bigcup_{i=1}^n W_i\right)\right\rangle \in \D$ and \\$\M \models_s \chi'(W_1 \ldots W_n)$.
\end{enumerate}
	\label{lemma:singleocc}
\end{Lemma}
\begin{proof}$\\$
\begin{description}
\item[1 $\Rightarrow$ 2] If such a $S$ exists, let $W_1 = W_2 = \ldots W_n = S$. Then since $\bigcup_i W_i = S$ we have that $\left \langle M, \left(\bigcup_i W_i\right) \right \rangle \in \D$, and since $\chi'(S \ldots S)$ is the same as $\chi(S)$ we have that $\M \models \chi'(W_1 \ldots W_n)$.
\item[2 $\Rightarrow$ 1] Suppose that such $W_1 \ldots W_n$ exist, and let $S = \bigcup_i W_i$. Then by assumption $\langle M, S\rangle \in \D$; and furthermore, since the $W_i$ occur only positively in $\chi'$ and $W_i \subseteq S$ for all $i$ we have that $\M \models \chi'(S \ldots S)$, that is, $\M \models \chi(S)$. 
\end{description}
\end{proof}

In this way, we managed to ``bring out'' all the dependencies of $\DD$ and convert the remaining expression to First Order Logic. Now we need to show that it is possible to ``put them back in''.\\

To do so, we will have to encode multiple relations into a single team. The obvious way to do so would be to fix tuples of variables $\tuple v_1 \ldots \tuple v_n$ and let a team $X$ correspond to the sets of tuples $X(\tuple v_i) = \{s(\tuple v_i) : s \in X\}$ for $i = 1 \ldots n$; but a problem with this is that, as long as the team $X$ is nonempty, it would not be possible to encode empty sets of tuples. So we will instead use \emph{two} tuples of variables\footnote{This is more than what is strictly necessary -- a tuple of variables and one extra variable to act as a flag would be enough -- but it makes the formulas somewhat simpler, and in this work the number of existential quantifiers is not a concern.} $\tuple v_i$ and $\tuple w_i$ to encode the relation $R_i = X\upharpoonright(\tuple v_i = \tuple w_i)(\tuple v_i) = \{s(\tuple v_i) : s \in X, s(\tuple v_i) = s(\tuple w_i)\}$: see Figure \ref{fig:reptuples} for an illustration. \\

\begin{figure}
\begin{align*}
& X = \begin{array}{| c | c c | c c c c | c c |}
\hline
& v_1 & w_1 & v_2 & v_3 & w_2 & w_3 & v_4 & w_4\\
\hline
s_0 & 0 & 0 & 0 & 1 & 0 & 1 & 0 & 1\\
s_1 & 1 & 1 & 1 & 2 & 1 & 2 & 0 & 1\\
s_2 & 2 & 2 & 0 & 0 & 1 & 1 & 0 & 1\\
\hline
\end{array}\\
& \\
& X\upharpoonright (v_1 = w_1)(v_1) = \{s_0,s_1, s_2\}(v_1) = \{0,1,2\}.\\
& X\upharpoonright (v_2 v_3 = w_2 w_3)(v_2 v_3) = \{s_0, s_1\}(v_2 v_3) = \{\langle 0, 1\rangle, \langle 1, 2\rangle\}.\\
& X\upharpoonright (v_4 = w_4)(v_4) = \emptyset(v_4) = \emptyset.
\end{align*}
\caption{A possible representation of the set of elements $A = \{0,1,2\}$, of the set of tuples $B = \{\langle 0, 1\rangle, \langle 1, 2\rangle\}$, and of the empty set $C = \emptyset$ in a single team. The variables $v_1$ and $w_1$ encode $A$, the variables $v_2 v_3$ and $w_2 w_3$ encode $B$, and the variables $v_4$ and $w_4$ encode $C$. In this way, it is possible to represent empty and nonempty relations within a single team.}
\label{fig:reptuples}
\end{figure}

Using this representation, it is first of all possible to state that the union of the sets of tuples of elements encoded by certain variables satisfy a certain dependency:
\begin{Lemma}
Let $\D$ be any dependency, let $k$ be its arity, let $\tuple v_1 \ldots \tuple v_n$, $\tuple w_1 \ldots \tuple w_n$ be tuples of variables such that all $\tuple v_i$ and $\tuple w_i$ have length equal to the arity of $\D$, and let $\D_\cup(\tuple v_1 \ldots \tuple v_n; \tuple w_1 \ldots \tuple w_n)$ be the $\FO(\D)$ formula 
\begin{align*}
	& \forall p_1 \ldots p_{n} \forall q \exists \tuple z_0 \tuple z_1 \left( \left(\bigvee_i q = p_i\right) \hookrightarrow \right. \\
& ~ ~ ~ ~ ~ ~ \left. \left( \bigwedge_{i \leq n} \left( \left (\bigwedge_{j < i} q \not = p_j \wedge q = p_i\right) \hookrightarrow \tuple z_0\tuple z_1 = \tuple v_i \tuple w_i\right)  \wedge (\tuple z_0 = \tuple z_1 \hookrightarrow \D \tuple z_0)\right)\right)
\end{align*}
	where $p_1 \ldots p_n$ and $q$ are new, distinct variables and $\tuple z_0$ and $\tuple z_1$ are tuples of new variables of the same length as the arity of $\D$. 

	Then for all suitable models $\M$ (with at least two elements) and teams $X$, $\M \models_X \D_\cup(\tuple v_1 \ldots \tuple v_n; \tuple w_1 \ldots \tuple w_n)$ if and only if $\left \langle M, \left(\bigcup_{i=1}^n W_i\right)\right\rangle \in \D$, where\\$W_i = X\upharpoonright (\tuple v_i = \tuple w_i) (\tuple v_i) = \{s(\tuple v_i) : s \in X, s(\tuple v_i) = s(\tuple w_i)\}$. 
	\label{lemma:union}
\end{Lemma}
\begin{proof}
	Suppose that $\M \models_X \D_\cup(\tuple v_1 \ldots \tuple v_n; \tuple w_1 \ldots \tuple w_n)$. Then there exist some $H$ and some
	\[
		Y = X[M/p_1 \ldots p_n q][H/\tuple z_0 \tuple z_1]\upharpoonright \left(\bigvee_i q = p_i\right)	\] 
such that
\begin{enumerate}
	\item For all $i = 1 \ldots n$ and for all $s \in X[M/p_1 \ldots p_n q]$, if $i$ is the smallest index such that $s(q) = s(p_i)$ then $H(s) = \{s(\tuple v_i \tuple w_i)\}$; 
	\item For $Z = Y \upharpoonright (\tuple z_0 = \tuple z_1)$, $\langle M,  Z(\tuple z_0)\rangle \in \D$. 
\end{enumerate}
	Now, $Y$ is precisely the set of all tuples in $X[M/p_1 \ldots p_n q][H/\tuple z_0 \tuple z_1]$ such that the value of $q$ is equal to the value of \emph{some} $p_i$. Therefore, the first condition implies that $Y\upharpoonright (\tuple z_0 = \tuple z_1)(\tuple z_0) = \bigcup_i (Y\upharpoonright (\tuple v_i = \tuple w_i)(\tuple v_i)) = \bigcup_i (X\upharpoonright (\tuple v_i = \tuple w_i)(\tuple v_i))$; and then the second condition implies that $\left \langle M, \left(\bigcup_i X\upharpoonright (\tuple v_i = \tuple w_i)(\tuple v_i)\right)\right\rangle \in \D$, as required.

	Conversely, suppose that $\left \langle M, \left(\bigcup_i X\upharpoonright (\tuple v_i = \tuple w_i)(\tuple v_i)\right)\right \rangle \in \D$. Then let $0$ and $1$ be two distinct elements in $M$, let $\tuple 0$ and $\tuple 1$ represent tuples of $k$ zeroes or ones, and let $H: X[M/p_1 \ldots p_n q]\rightarrow \parts(M^{2k}) \backslash \{\emptyset\}$ be defined as 
\[
	H(s) = \left\{\begin{array}{l l}
		\{s(\tuple v_i \tuple w_i)\} & \text{ if } i \text{ is the smallest index such that } s(q) = s(p_i);\\
		\{\tuple 0\tuple 1\} & \text{ if } s(q) \not = s(p_i) ~\forall i \in \{1 \ldots n\}.
	\end{array}
\right.
\]
	Then let $Y = X[M/p_1 \ldots p_n q][H/\tuple z_0 \tuple z_1]\upharpoonright \left(\bigvee_i q = p_i\right)$. By construction, it is clear that 
	\[
		\M \models_Y \left (\bigwedge_{j < i} q \not = p_j \wedge q = p_i \right) \hookrightarrow \tuple z_0 \tuple z_1 = \tuple v_i \tuple w_i
	\]
	for all $i = 1 \ldots n$. 
	Moreover, $Y(\tuple z_0 \tuple z_1) = \bigcup_i X(\tuple v_i \tuple w_i)$, because for each $s \in X$ and for each $i = 1 \ldots n$ there exists an assignment $s^i = s[0 \ldots 1 \ldots 1/p_1 \ldots p_n][1/q] \in X[M/p_1 \ldots p_n q]$ such that $s^i(p_j)$ is $0$ if $j < i$ and $1$ otherwise and such that $s^i(q) = 1$, and for such an assignment we have that $H(s^i) = \{s^i(\tuple v_i \tuple w_i)\} = \{s(\tuple v_i\tuple w_i)\}$, as required. Thus we also have that $Y\upharpoonright (\tuple z_0 = \tuple z_1)(\tuple z_0) = \bigcup_i X\upharpoonright (\tuple v_i = \tuple w_i)(\tuple v_i)$. But by assumption we know that\\ $\left \langle M, \bigcup_i X\upharpoonright (\tuple v_i = \tuple w_i)(\tuple v_i)\right \rangle \in \D $; therefore $\left \langle M, Y\upharpoonright(\tuple z_0 = \tuple z_1)(\tuple z_0)\right \rangle \in \D$, and thus $\M \models_Y \tuple z_0 = \tuple z_1 \hookrightarrow \D \tuple z_0$, and finally $\M \models_X \D_\cup(\tuple v_1 \ldots \tuple v_n; \tuple w_1 \ldots \tuple w_n)$ as required.
\end{proof}

Finally, we need a way to ``put back'' the dependencies of $\DD$ into the formula. The next two lemmas take care of that: 
\begin{Lemma}
	Let $\psi(W)$ be a \textbf{quantifier-free} $\FO(\DD)$ formula in which the $k$-ary relation symbol $W$ occurs at most once and only positively and let $\psi'(\tuple v, \tuple w)$ be the formula obtained by replacing the (unique) occurrence $W \tuple t$ of $W$ in $\psi$ with $\tuple v = \tuple w \wedge \tuple t = \tuple w$, where $\tuple v$ and $\tuple w$ are tuples of new variables of length equal to the arity of $W$. Then, for all suitable models $\M$, relations $W$ and teams $X$,  $\M \models_X \psi(W)$ if and only if there exists some $H$ such that 
	\begin{enumerate}
		\item $\M \models_{X[H/\tuple v \tuple w]} \psi'(\tuple v, \tuple w)$; 
		\item $X[H/\tuple v \tuple w]\upharpoonright (\tuple v = \tuple w)(\tuple v) \subseteq W$. 
	\end{enumerate}
	\label{lemma:qfree}
\end{Lemma}
\begin{proof}
	The proof is by structural induction on $\psi$. 
	\begin{itemize}
		\item If $\psi$ is of the form $W \tuple t$, $\psi'$ is simply $\tuple v = \tuple w \wedge \tuple t = \tuple w$. Now if $\M \models_X W \tuple t$ it must be that $s(\tuple t) \in W$ for all $s \in X$. Thus, if we define $H$ so that $H(s) = \{s(\tuple t) s(\tuple t)\}$ for all $s \in X$ we have at once that $X[H/\tuple v \tuple w]\upharpoonright (\tuple v = \tuple w) (\tuple v) = \{s(\tuple t): s \in X\} \subseteq W$, and furthermore $\M \models_{X[H/\tuple v \tuple w]} \tuple v = \tuple w \wedge \tuple t = \tuple w$ as required.

			Conversely, suppose that there exists some $H$ such that $X[H/\tuple v \tuple w]\upharpoonright (\tuple v = \tuple w)(\tuple v) \subseteq W$ and $\M \models_{X[H/\tuple v \tuple w]} \tuple v = \tuple w \wedge \tuple t = \tuple w$. Now, for any $s \in X$, let $\tuple m = s(\tuple t)$: then $s[\tuple m \tuple m/\tuple v \tuple w] \in X[H/\tuple v \tuple w]\upharpoonright (\tuple v = \tuple w)$, and therefore $\tuple m \in W$.  
Thus $\M \models_X W \tuple t$ as required. 
		\item If $\psi(W)$ is of the form $\psi_1 \vee \psi_2$, let us assume -- without loss of generality -- that $W$ occurs only in $\psi_1$. Then $\psi'(\tuple v, \tuple w)$ is $\psi_1'(\tuple v, \tuple w) \vee \psi_2$. Now suppose that $\M \models_X \psi_1(W) \vee \psi_2$. Then $X = Y \cup Z$ for two $Y$, $Z$ such that $\M \models_Y \psi_1(W)$ and $\M \models_Z \psi_2$. So by induction hypothesis we know that there exists some $H: Y \rightarrow \parts(M^{2k}) \backslash \{\emptyset\}$ such that $\M \models_{Y[H/\tuple v \tuple w]} \psi_1'(\tuple v, \tuple w)$ and $Y[H/\tuple v \tuple w]\upharpoonright (\tuple v = \tuple w)(\tuple v) \subseteq W$. Now let $\tuple m_1$ and $\tuple m_2$ be two arbitrary, distinct tuples of elements, and let $H': X \rightarrow \parts(M^{2k})\backslash \{\emptyset\}$ be defined as 
			\[
				H'(s) = \left\{\begin{array}{l l}
					H(s) \cup \{\tuple m_1 \tuple m_2\} & \text{ if } s \in Y \cap Z;\\
					H(s) & \text{ if } s \in Y \backslash Z;\\
					\{\tuple m_1 \tuple m_2\} & \text{ if } s \in Z \backslash Y.
				\end{array}
					\right.
			\]
			Then $X[H'/\tuple v \tuple w] = Y[H/\tuple v \tuple w] \cup Z[\tuple m_1 \tuple m_2 / \tuple v \tuple w]$, $\M \models_{Z[\tuple m_1 \tuple m_2/\tuple v \tuple w]} \psi_2$ by locality, and $X[H'/\tuple v \tuple w]\upharpoonright (\tuple v = \tuple w)(\tuple v)  = Y[H/\tuple v \tuple w]\upharpoonright (\tuple v = \tuple w)(\tuple v) \subseteq W$. Moreover $\M \models_{X[H'/\tuple v \tuple w]} \psi'_1(\tuple v, \tuple w) \vee \psi_2$, as required. 

			Conversely, suppose that there is some $H$ such that \\$\M \models_{X[H/\tuple v \tuple w]} \psi'_1(\tuple v, \tuple w) \vee \psi_2$ and such that $X[H/\tuple v \tuple w]\upharpoonright (\tuple v = \tuple w)(\tuple v) \subseteq W$. Then $X[H/\tuple v\tuple w] = Y' \cup Z'$ for two $Y'$, $Z'$ such that $\M \models_{Y'} \psi'_1(\tuple v, \tuple w)$ and $\M \models_{Z'} \psi_2$. Now let $Y = \{s \in X: \exists \tuple m_1 \tuple m_2 \text{ s.t. } s[\tuple m_1 \tuple m_2/ \tuple v \tuple w] \in Y'\}$ and let $H': Y \rightarrow \parts(M^{2k}) \backslash \{\emptyset\}$ be defined as 
			\[
				H'(s) = \{\tuple m_1 \tuple m_2 \in M^{2 k} : s[\tuple m_1 \tuple m_2/\tuple v \tuple w] \in Y'\}.
			\]
			Then by construction $H'(s) \not = \emptyset$ for all $s \in Y$, and furthermore $Y[H'/\tuple v \tuple w] = Y'$; therefore, $\M \models_{Y[H'/\tuple v \tuple w]} \psi'_1(\tuple v, \tuple w)$, and since $Y' \subseteq X[H/\tuple v \tuple w]$ we have that  
\begin{align*}
Y[H'/\tuple v\tuple w] \upharpoonright (\tuple v = \tuple w)(\tuple v) &= 
Y'\upharpoonright(\tuple v = \tuple w)(\tuple v)\\
& \subseteq 
X[H/\tuple v\tuple w]\upharpoonright (\tuple v = \tuple w)(\tuple v) \subseteq W.
\end{align*}

 Then by induction hypothesis we have that $\M \models_Y \psi_1(W)$. Now let $Z = \{s \in X : \exists \tuple m_1 \tuple m_2 \text{ s.t. } s[\tuple m_1 \tuple m_2/\tuple v \tuple w] \in Z'\}$. By locality, we have that $\M \models_Z \psi_2$, and moreover $X = Y \cup Z$ (since if $s \in X$ then there are some $\tuple m_1 \tuple m_2$ such that $s[\tuple m_1 \tuple m_2/\tuple v \tuple w] \in X[H/\tuple v \tuple w] = Y' \cup Z'$) and hence $\M \models_X \psi_1(W) \vee \psi_2$ as required. 

		\item If $\psi(W)$ is of the form $\psi_1 \wedge \psi_2$, let us again assume that $W$ occurs only in $\psi_1$. Then $\psi'(\tuple v, \tuple w) = \psi_1'(\tuple v, \tuple w) \wedge \psi_2$. Suppose that $\M \models_X \psi_1(W) \wedge \psi_2$: then $\M \models_X \psi_1(W)$ and $\M \models_X \psi_2$, and by induction hypothesis there exists some $H$ such that $X[H/\tuple v \tuple w]\upharpoonright (\tuple v = \tuple w)(\tuple v) \subseteq W$ and $\M \models_{X[H/\tuple v \tuple w]} \psi'_1(\tuple v, \tuple w)$. Furthermore, by locality, we have that $\M \models_{X[H / \tuple v \tuple w]} \psi_2$, and therefore $\M \models_{X[H/\tuple v \tuple w]} \psi'_1 (\tuple v,\tuple w)\wedge \psi_2$ as required. 

			Conversely, suppose that there is some $H$ such that $\M \models_{X[H/\tuple v \tuple w]} \psi'_1(\tuple v, \tuple w) $\\$\wedge \psi_2$ and $X[H/\tuple v \tuple w]\upharpoonright (\tuple v = \tuple w)(\tuple v) \subseteq W$. Then by locality $\M \models_X \psi_2$ and by induction hypothesis $\M \models_{X} \psi_1(W)$, and hence $\M \models_X \psi_1(W) \wedge \psi_2$ as required. 
	\end{itemize}
\end{proof}
\begin{Lemma}
	Let $\chi(W_1 \ldots W_q) = \forall \tuple x_1 \exists \tuple y_1 \ldots \forall \tuple x_n \exists \tuple y_n \psi(W_1 \ldots W_q)$ be a $\FO(\DD)$ sentence in Quantifier Normal Form in which the $W_i$ -- which all have the same arity $k$ -- occur only positively and at most once each, and let $\D$ be a downwards closed dependency of the same arity of the $W_i$. Then the following are equivalent: 
	\begin{enumerate}
		\item There exist $W_1 \ldots W_q$ such that $\left \langle M, \bigcup_i W_i\right \rangle \in \D$ and $\M \models \chi(W_1 \ldots W_q)$; 
		\item $\M \models \forall \tuple x_1 \ldots \exists \tuple y_n \exists \tuple v_1 \tuple w_1 \ldots \tuple v_q \tuple w_q (\D_\cup(\tuple v_1 \ldots \tuple v_q; \tuple w_1 \ldots \tuple w_q) \wedge $\\$\psi'(\tuple v_1 \tuple w_1 \ldots \tuple v_q \tuple w_q))$, where $\psi'(\tuple v_1 \tuple w_1 \ldots \tuple v_q \tuple w_q)$ is obtained -- as in the previous lemma -- by replacing each $W_i \tuple t_i$ with $\tuple v_i = \tuple w_i \wedge \tuple t_i = \tuple w_i$ for tuples of new variables $\tuple v_i$ and $\tuple w_i$. 
	\end{enumerate}
	\label{lemma:putback}
\end{Lemma}
\begin{proof}
	Suppose that 1. holds. Then there exist  $W_1 \ldots W_q$ such that\\$\left\langle M, \bigcup_i W_i\right \rangle \in \D$ and $\M \models \chi(W_1 \ldots W_q)$. Then in particular there exist functions $H_1 \ldots H_n$ such that, for $X = \{\epsilon\}[M / \tuple x_1] [H_1 / \tuple y_1] \ldots[M/\tuple x_n][H_n/\tuple y_n]$, $\M \models_X \psi(W_1 \ldots W_q)$. Applying repeatedly Lemma \ref{lemma:qfree}, we can find $K_1 \ldots K_q$ such that, for $Y = X[K_1/\tuple v_1 \tuple w_1]\ldots[K_q/\tuple v_q \tuple w_q]$,
	\begin{enumerate}[a)]
		\item $\M \models_{Y} \psi'(\tuple v_1 \tuple w_1 \ldots \tuple v_q \tuple w_q)$; 
		\item For all $i \in 1 \ldots q$, $Y\upharpoonright (\tuple v_i = \tuple w_i)(\tuple v_i) \subseteq W_i$.
	\end{enumerate}
	Now, for all $i$, let $W'_i  = Y\upharpoonright (\tuple v_i = \tuple w_i)(\tuple v_i)$. Then since $\D$ is downwards closed and $\bigcup_i W'_i \subseteq \bigcup_i W_i$ we have that $\M \models_Y \D(\bigcup_i W'_i)$, and hence $\M \models_Y \D_\cup(\tuple v_1 \ldots \tuple v_q ; \tuple w_1 \ldots \tuple w_q)$ by Lemma \ref{lemma:union}. This proves that 
	\[
		\M \models_X \exists \tuple v_1 \tuple w_1 \ldots \tuple v_q \tuple w_q (\D_\cup(\tuple v_1 \ldots \tuple v_q; \tuple w_1 \ldots \tuple w_q) \wedge \psi'(\tuple v_1 \tuple w_1 \ldots \tuple v_q \tuple w_q))
	\]
	and hence that 2. holds. \\

	Conversely, suppose that 2. holds. Then there exist functions $H_1 \ldots H_n$, $K_1 \ldots K_q$ such that, for $X = \{\epsilon\}[M/\tuple x_1][H_1/\tuple y_1]\ldots[M/\tuple x_n][H_n/\tuple y_n]$ and for $Y = X[K_1 /\tuple v_1 \tuple w_1]\ldots [K_q/\tuple v_q \tuple w_q]$, $\M \models_Y \D_\cup(\tuple v_1 \ldots \tuple v_q; \tuple w_1 \ldots \tuple w_q) \wedge \psi'(\tuple v_1 \tuple w_1 \ldots \tuple v_q \tuple w_q)$. Now, for each $i = 1 \ldots q$, let $W_i = Y\upharpoonright (\tuple v_i = \tuple w_i)(\tuple v_i)$. Then by Lemma \ref{lemma:union} we have that $\left \langle M, \bigcup_i W_i \right \rangle \in \D$; and furthermore, by Lemma \ref{lemma:qfree} $\M \models_Y \psi(W_1 \ldots W_n)$. But then by locality $\M \models_X \psi(W_1 \ldots W_n)$, and hence $\M \models \chi(W_1 \ldots W_n)$ and 1. holds. 
\end{proof}

We can now finally prove the main result of this section: 
\begin{Theorem}
	Let $\DD$ be any family of downwards closed dependencies and let $\SSd$ be any strongly first order family of dependencies. Then $\SSd$ is safe for $\DD$. 
\label{thm:safe}
\end{Theorem}
\begin{proof}
	Let $\phi$ be a sentence of $\FO(\SSd, \DD)$. Then, by Lemmas \ref{lemma:removeSFO} and \ref{lemma:singleocc}, there exist a first order formula $\theta(W_1^1 \ldots W_1^{r(1)}, \ldots, W_k^1 \ldots W_k^{r(k)})$, in which each $W_i^j$ occurs only once, and dependencies $\D_{j(1)} \ldots \D_{j(k)}\in \DD$ such that the following are equivalent for all models $\M$:
	\begin{enumerate}[A)]
	\item  $\M \models \phi$; 
	\item There exist relations $W_1^1 \ldots W_1^{r(1)}, \ldots, W_k^1 \ldots W_k^{r(k)}$ such that
	\begin{itemize}
		\item For all $i = 1 \ldots k$, $\left \langle M, \bigcup_{t=1}^{r(i)} W_i^t\right\rangle \in \D_{j(i)}$;
		\item $\M \models \theta(W_1^1 \ldots W_1^{r(1)}, \ldots, W_k^1 \ldots W_k^{r(k)})$. 
	\end{itemize}
	\end{enumerate}

	Now let $\theta(W_1^1 \ldots W_1^{r(1)}, \ldots, W_k^1 \ldots W_k^{r(k)})$ have Quantifier Normal Form 
	\[
		\forall \tuple x_1 \exists \tuple y_1 \ldots \forall \tuple x_l \exists \tuple y_l \psi(W_1^1 \ldots W_1^{r(1)}, \ldots, W_k^1 \ldots W_k^{r(n)})
	\]
	But applying repeatedly Lemma \ref{lemma:putback} and the fact that all $\D_{j(i)}$ are downwards closed\footnote{For the applications of Lemma \ref{lemma:putback} after the first one we can treat the $(\D_{j(i)})_\cup$ already introduced as ordinary dependency atoms, so that the expression is still in Quantifier Normal Form.} we can see that there exists some $\FO(\DD)$ sentence 
\begin{align*}
\phi^* := & \forall \tuple x_1 \exists \tuple y_1 \ldots \forall \tuple x_l \exists \tuple y_l ~  \exists \tuple v_1^1 \tuple w_1^1 \ldots \tuple v_1^{r(1)} \tuple w_1^{r(1)} \ldots \tuple v_k^1 \tuple w_k^1 \ldots \tuple v_k^{r(k)} \tuple w_k^{r(k)}\\
&  (
(\bigwedge_i (\D_{j(i)})_\cup(\tuple v_i^1 \ldots \tuple v_i^{r(i)}, \tuple w_i^1 \ldots \tuple w_i^{r(i)})) \wedge\\
& ~~\psi'(\tuple v_1^1 \tuple w_1^1 \ldots \tuple v_1^{r(1)} \tuple w_1^{r(1)} \ldots 
\tuple v_k^1 \tuple w_k^1 \ldots \tuple v_k^{r(k)} \tuple w_k^{r(k)})
)
\end{align*}
 which is true if and only if B) above holds. Thus, every $\FO(\SSd, \DD)$ sentence is equivalent to some $\FO(\DD)$ sentence - in other words, $\SSd$ is safe for $\DD$. 
\end{proof}
\section{Characterizing strongly first order, relativizable, downwards\\ closed dependencies}
We can now prove that all dependencies that are nontrivial, relativizable, downwards closed and strongly first order are definable in terms of constancy. To do so, we will need to use the Chang-Makkai Theorem, applying it not to unary relations (as is presented e.g. in Theorem 5.3.6 of \cite{chang1990model}) but to $k$-ary relations (as mentioned in \cite{chang1964some}, the proof carries over to this case without problems. For reference, the Appendix contains the detailed proof of the case that interests us). On the other hand, it suffices to consider countable structures and vocabularies containing only a relation symbol $R$. 

So this is the variant of the Chang-Makkai Theorem that we need: 
\begin{Theorem}
	Let $\Phi(R)$ be a first order sentence on the vocabulary $\{R\}$, where $R$ is $k$-ary. The following are equivalent: 
	\begin{enumerate}
		\item For every countable model $\M$ over the signature $\{R\}$, $|\{R : R \subseteq M^k, \M \models \Phi(R)\}| < \aleph_1$; 
		\item There are a finite number of formulas $\theta_1(\tuple x, \tuple z) \ldots \theta_n(\tuple x, \tuple z)$, over the empty vocabulary, such that 
			\[
				\phi(R) \models \bigvee_{i=1}^n \exists \tuple z \forall \tuple x (R \tuple x \leftrightarrow \theta_i(\tuple x, \tuple z)).
			\]
	\end{enumerate}
	\label{thm:changmakkai}
\end{Theorem}
The idea of this section's main proof is to show that if $\D$ is strongly first order then the property $\Dmax$ of being maximal among the $R$ that satisfy $\D$ is also strongly first order; that whenever $\langle M, R\rangle \in \D$ there is at least one $R' \supseteq R$ such that $\langle M, R'\rangle \in \Dmax$; and that for every countable model there must be a countable number of such maximal $R'$. Then we use the above version of the Chang-Makkai Theorem to show that these maximal sets are definable over the empty signature, and use the downwards closure property to define the $\D$ atoms via Boolean disjunctions (that is, unraveling Definition \ref{def:booldisj}, via constancy atoms). 

\begin{Proposition}
	Let $\D$ be a downwards closed, strongly first order dependency. Then there exists a first order dependency $\Dmax$ such that $\langle M, R\rangle \in \Dmax$ if and only if $\langle M, R\rangle \in \D$ and there is no $S \supsetneq R$ such that $\langle M, S\rangle \in \D$. Furthermore, $\Dmax$ is strongly first order itself.
	\label{propo:DmaxFO}
\end{Proposition}
\begin{proof}
	Since $\D$ is strongly first order, the $\FO(\D)$ sentence 
	\[
		\phi(R) := \exists \tuple x (\lnot R \tuple x \wedge \forall \tuple y ( (\lnot R \tuple y \wedge \tuple x \not = \tuple y) \vee \D \tuple y))
	%
	\]
	is equivalent to some first order sentence. Observe that $\M \models \phi(R)$  if and only if there exists some $S$ such that $R \subsetneq S$ and $\langle M, S\rangle \in \D$. Indeed, suppose that such a $S$ exists and let $\tuple a \in S \backslash R$. Then, choosing $\tuple a$ as the only value for $\tuple x$,\footnote{More precisely: picking $\tuple x$ according to the function $H: \{\epsilon\} \rightarrow \parts(M) \backslash \{\emptyset\}$ such that $H(\epsilon) = \tuple a$.} we have that $\M \models_{\{(\tuple x : \tuple a)\}} \lnot R \tuple x$. Now let $X = \{(\tuple x: \tuple a)\}[M/\tuple y]$, and split it into 
	$Z = \{s \in X : s(\tuple y) \in S\}$ and $Y = X \backslash Z = \{s \in X : s(\tuple y) \not \in S\}$. Clearly $\M \models_Z \D \tuple y$, because $Z(\tuple y) = S$. Moreover, for all $s \in Y$ we have that $s(\tuple y) \not \in R$, because $s(\tuple y) \not \in S$ and $R \subseteq S$, and that $s(\tuple y) \not = s(\tuple x)$, because $s(\tuple x) = \tuple a \in S$. Thus we have that $\M \models_Y \lnot R \tuple y \wedge \tuple x \not = \tuple y$, as required. 

	Conversely, suppose that $\M \models_{\{\epsilon\}}  \phi(R)$, where $\epsilon$ is the empty assignment. Then there exists some $X = \{\epsilon\}[K/\tuple x][M/\tuple y]$ such that $X(\tuple x) \cap R = \emptyset$ and $\M \models_X (\lnot R \tuple y \wedge \tuple x \not = \tuple y) \vee \D \tuple y$. Thus, $X = Y \cup Z$ for two $Y$, $Z$ such that $\M \models_Y (\lnot R \tuple y \wedge \tuple x \not = \tuple y)$ and $\M \models_Z \D \tuple y$. Let $S = Z(\tuple y)$: then $\langle M, S\rangle \in \D$. Furthermore, $R$ is contained in $S$: indeed, for all $\tuple m \in R$ there is some $s \in X$ is such that $s(\tuple y) = \tuple m$, and then $s \in Z$, because $\M \models_Y \lnot R \tuple y$ and $X = Y \cup Z$. Finally, $S$ is not contained in $R$: indeed, if $\tuple a$ is a possible value for $\tuple x$ in $X$ then $\tuple a \not \in R$ and $\tuple a \in S$ (because the assignment $(\tuple x \tuple y: \tuple a \tuple a)$ is in $X$ but not in $Y$).
%

	If $\phi(R)$ is equivalent to a first order sentence, the same is true of its negation. Let $\psi(R)$ be the first order sentence equivalent to the negation of $\phi(R)$: then $\M \models \psi(R)$ if and only if there is no $S \supsetneq R$ such that $\langle M, S\rangle \in \D$.
	Thus $\Dmax(R)$ is equivalent to $\D(R) \wedge \psi(R)$, and therefore it is first order. It remains to show that it is \emph{strongly} first order as a dependency.\\

	 Now $\psi(R)$ is upwards closed in $R$: indeed, if there is no $S \supsetneq R$ such that $\D(S)$ and $R \subseteq R'$ then clearly there is no $S \supsetneq R'$ such that $\D(S)$ either. By Theorem \ref{thm:upwards}, upwards closed dependencies are strongly first order; therefore, the dependency $\E(R) = \{\langle M, R\rangle: \langle M, R\rangle \models \psi(R)\}$ is strongly first order. 

 But by Theorem \ref{thm:safe}, strongly first order dependencies are safe for downwards closed dependencies. So in particular $\E$ is safe for $\D$, and the atom $\Dmax(\tuple t)$ is definable in $\FO(\D, \E)$ as $\Dmax(\tuple t) := \D \tuple t \wedge \E\tuple t$. Therefore every sentence $\phi \in \FO(\Dmax)$ is equivalent to some sentence $\phi' \in \FO(\D, \E)$, which by the safety of $\E$ for $\D$ is equivalent to some sentence $\phi'' \in \FO(\D)$, which by the strongly first orderness of $\D$ is equivalent to some first order sentence. Therefore $\Dmax$ is strongly first order, as required.
\end{proof}

The following consequence of the above proposition will also be necessary for our proof:
\begin{Corollary}
	Let $\D(R)$ be a downwards closed, strongly first order dependency and let $M, R$ be such that $\langle M, R\rangle \in \D$. Then there exists a $S \supseteq R$ such that $\langle M, S\rangle \in \Dmax$.
	\label{coro:maximal}
\end{Corollary}
\begin{proof}
	Consider the $\FO(\Dmax)$ sentence 
	\[
		\theta(T) := \forall \tuple x(\lnot T \tuple x \vee \Dmax \tuple x)
	\]
	It is easy to check that $\M \models \theta(T)$ if and only if there exists some $K \supseteq T$ such that $\langle M, K\rangle \in \Dmax$; and since, as we just saw, $\Dmax$ is strongly first order, $\theta(T)$ is equivalent to some first order sentence. Let $\chi(T)$ be the first order sentence equivalent to the negation of $\theta(T)$: then $\M \models \chi(T)$ if and only if there is no $K \supseteq T$ such that $\langle M, K\rangle \in \Dmax$. Furthermore, $\chi(T)$ is clearly upwards closed in $T$, since if there are no $K \supseteq T$ with $\langle M, K \rangle \in \Dmax$ and $T \subseteq T'$ then there are no $K \supseteq T'$ with $\langle M, K\rangle \in \Dmax$ either. 

	Therefore by Theorem \ref{thm:upwards} the dependency $\F = \{\langle M, T\rangle : \langle M, T\rangle \models \chi(T)\}$ is strongly first order. Consider now then the $\FO(\D, \F)$ sentence on the empty signature
	\[
		\exists \tuple x (\D \tuple x \wedge \F\tuple x).
	\]
	By construction, this is true if and only if there exists a relation $R$ which satisfies $\D(R)$ but is contained in no maximal superset; and since strongly first order dependencies are safe for downwards closed dependencies by Theorem \ref{thm:safe}, this is equivalent to some $\FO(\D)$ sentence and thus -- since $\D$ is strongly first order itself -- to some first order sentence. Therefore, there exists a first order sentence $\eta$ over the empty vocabulary such that $\M \models \eta$ if and only if there is some relation $R$ with domain $M$ which satisfies $\D(R)$ but which is not contained in any $R'$ satisfying $\Dmax(R)$. 

	This is clearly not true if the model $\M$ is finite, so for all finite models $\M$ we have that $\M \models \lnot \eta$. But then, since $\lnot \eta$ is a first order sentence over the empty signature, by compactness the same is true for all models, finite or infinite. 

	By the semantics of $\eta$ this implies that whenever $\langle M, R\rangle \in \D$ there exists some $R' \supseteq R$ such that $\langle M, R'\rangle \in \Dmax$, as required.
\end{proof}

As a quick aside, it is perhaps worth briefly pointing out here that this corollary fails for arbitrary first order dependencies: if $\langle M, R\rangle \in \D$, it is not necessarily the case that there exists some maximal $R' \supseteq R$ such that $\langle M, R'\rangle \in \D$. As a counterexample, let $\D(R)$ be the 4-ary dependency which is defined by the conjunction of the following axioms: 
\begin{itemize}
\item If we write ``$x \leq y$'' for $\exists z u R x y z u$, the relation $\leq$ is a linear order with endpoints; 
\item If we write ``$B z$'' for $\exists x y u R x y z u$, $B$ is not true of the starting point of the linear order, it is true of some element, and whenever it is true of some element it is true of its immediate predecessor along $\leq$ (if any exists); 
\item If we write ``$T u$'' for $\exists x y z R x y z u$, there exists some element $a$ such that $\lnot B a$ and such that for all $b$, $T b$ if and only if $b \leq a$ in the above order.
\end{itemize}
Then consider a model $\M$ with domain $\mathbb N \cup \{\infty\}$ and the team $X$  with domain $\{x,y,z,w\}$ defined as 
\begin{align*}
	&\{s :  s(x) \leq s(y) \text{, } s(z) = \infty \text{ and } s(u) \leq 1\}
\end{align*}
corresponding to the graphical representation
\begin{center}
\includegraphics{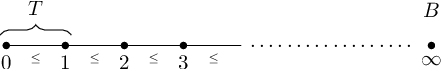}
\end{center}

Then it is easy to check that $\langle M, X(\tuple t)\rangle \in \D$, but no maximal $R' \supseteq X(\tuple t)$ satisfying $\D$ exists: the range of $T$ (that is, the projection of the relation over the fourth element) can be extended to any set of the form $\{n \in \mathbb N: n \leq m\}$ for any $m \in \mathbb N$ but not to $\mathbb N$ or to $\mathbb N \cup \{\infty\}$, while nothing can be added to the linear order $\leq$ (that is, to the projection over the first two elements) or to $B$ (that is, to the projection over the third element) without violating our axioms. 

The fact that whenever $\M \models_X \D \tuple t$ there exists some maximal $X' \supseteq X$ such that $\M \models_{X'} \D \tuple t$, therefore, is one that does not hold for all first order dependencies (although it is easy to check that it holds for functional dependence, independence and inclusion, and although we just proved that it holds for all downwards closed strongly first order dependencies).

\begin{Proposition}
	Let $A = \{a_1, a_2, \ldots\}$ be a countable set of elements and let $\D$ be a strongly first order, downwards closed, relativizable dependency of arity $k$. Then there exist only countably many $R \subseteq A^k$ such that $\langle A, R\rangle \in \Dmax$.
	\label{propo:countable}
\end{Proposition}
\begin{proof}
	Suppose that there exist uncountably many such $R$. Then let $I$ be a $(k+1)$-ary relation symbol, let $P$ be a unary symbol, and consider the first order theory $T$ containing the following axioms: 
	\begin{enumerate}[(a)]
		\item $\exists x_1 \ldots x_n \left(\bigwedge_{i<j\leq n} x_i \not = x_j \wedge \bigwedge_{i=1}^n Px_i\right)$, for all $n \in \mathbb N$;
		\item $\forall \tuple z \forall q (I \tuple z q \rightarrow \bigwedge_{z_i \in \tuple z} Pz_i)$;
		\item $\forall q \Dmax^P(I\h q)$, where $\Dmax^P(I\h q)$ is obtained from $\Dmax^P(R)$ (that is, from the relativization of $\Dmax(R)$ to $P$) by replacing every $R \tuple v$ with $I \tuple v q$; 
		\item $\lnot \chi(I, P)$, where $\chi(I, P)$ is the first order sentence equivalent\footnote{Such a first order sentence exists because $\D$ is strongly first order \emph{and} relativizable. Note that we do not need $\Dmax$ to be relativizable -- for (c) we only need that $\Dmax^P(R)$ is a first order sentence, which is certainly true since $\Dmax(R)$ is a first order sentence -- but just that $\D$ is relativizable.} to the $\FO(\D^P)$ sentence $\forall q \exists \tuple z( \lnot I \tuple z q \wedge \D^P \tuple z)$.
	\end{enumerate}
	I state that if our assumption is true, the above theory  -- that is first order and has a finite vocabulary -- has only uncountable models. This is clearly impossible by the L\"owenheim-Skolem Theorem, and therefore our assumption cannot hold.
	
	Indeed, suppose that $\M = \langle M, P, I\rangle$ is a countable model of $T$. Then by (a) $P$ is infinite and countable, and we can assume that it is $A$ up to isomorphism; by (b) and (c), for every $m \in M$ the relation $I \h m = \{\tuple a \in M^k: \langle \tuple a, m\rangle \in I\}$ is in $P^k$ and satisfies $\Dmax^P$ in $M$, that is to say satisfies $\Dmax$ in $P$; and by (d), as we will see, \emph{every} relation that satisfies $\Dmax$ in $P$ is equal to $I \h m$ for some $m \in M$. But we said that $A$ -- and, therefore, $P$ -- contains uncountably many distinct $R$ which satisfy $\Dmax(R)$ in it, and therefore this is impossible. 

	On the other hand -- again, postponing the verification that (d) holds if and only if $I$ enumerates all $R$ that satisfy $\Dmax(R)$ in $P$ -- there certainly exists an uncountable model $\M = \langle M, P, I\rangle$ of the above theory: let $M$ be a set of cardinality $2^{\aleph_0}$ containing $A$, let the interpretation of $P$ be $A$ itself, and let $I\h m$ range over all (uncountably many, but certainly no more than $|M| = 2^{|A|}$) $R \subseteq P^k$ such that $\langle P, R\rangle \in \Dmax$ as $m$ ranges over $M$. Thus our theory has an uncountable model but no countable models, which is impossible.\\

	We now verify that (d) holds if and only if $I\h m$ enumerates all relations that satisfy $\Dmax$ with respect to $P$, that is, that the $\FO(\D^P)$ sentence $\xi := \forall q \exists \tuple z( \lnot I \tuple z q \wedge \D^P \tuple z)$ is true if and only if $I\h m$ does \emph{not} enumerate all such subsets as $m$ ranges over $M$. 

	\begin{itemize}
		\item \textbf{Left to right:} Suppose that $\M \models \xi$. Then there exists a team $X = \{\epsilon\}[M/q][H/\tuple z]$, for some function $H$, such that $\M \models_X \lnot I \tuple z q \wedge \D^P\tuple z$. Now let $R = X(\tuple z)$. From the second conjunct, we have at once that $R \subseteq P^k$ and that $\langle P, R\rangle \in \D$; and for every $q \in M$ there exists some assignment $s \in X$ for which $s(q) = m$, and thus from the first conjunct we have that $\langle s(\tuple z), m\rangle \not \in I$, that is $s(\tuple z) \not \in I\h m$. Thus, $\langle P, R\rangle \in \D$ but $R$ is not contained in any $I \h m$; and since by Corollary \ref{coro:maximal} this $R$ must be contained in some relation that satisfies $\Dmax$ in $P$, it follows that $I\h m$ does not enumerate all maximal relations.

		\item \textbf{Right to left:} Suppose that $I\h m$ does not enumerate all subsets of $P^k$ which satisfy $\Dmax$ in $P$ for $m$ ranging in $M$. Then in particular there exists some $R \subseteq P^k$ such that $\langle P, R\rangle \in \Dmax$ and $R \backslash I\h m \not = \emptyset$ for all $m \in M$. Now consider the following function $H: \{\epsilon\}[M/q] \rightarrow \parts(M^k) \backslash \{\emptyset\}$: 
			\[
				H(s) = R \backslash I \h (s(q)) = \{\tuple a \in R : \langle \tuple a, s(q)\rangle \not \in I\}.
			\]
			As we just saw, $H(s) \not = \emptyset$ for all $s \in \{\epsilon\}[M/q]$. For $X = \{\epsilon\}[M/q][H/\tuple z]$, we clearly have that $\M \models_X \lnot I \tuple z q$. Furthermore, $\M \models_X \D^P \tuple z$: indeed, $X(\tuple z) \subseteq R$ by construction, $\langle P, R\rangle  \in \D$ since $\langle P, R\rangle \in \Dmax$, thus $\langle P, X(\tuple z)\rangle \in \D$ by downwards closure and finally $\M \models_X \D^P \tuple z$. Therefore $\M \models \xi$ as required.
	\end{itemize}
\end{proof}
At this point, our conclusion follows from the Chang-Makkai Theorem. 
\begin{Theorem}
	Let $\D(R)$ be a downwards closed, strongly first order, relativizable dependency that has the empty team property (that is, it is satisfied by the empty team)\footnote{The only downwards closed dependency that does not have this property is the trivial one that is false for all teams.}. Then it is definable in $\FO(=\!\!(\cdot))$. 
\end{Theorem}
\begin{proof}
	By Proposition \ref{propo:countable}, $|\{R : R \subseteq M^k, \langle M, R\rangle  \in \Dmax\}|$ is countable for all countable models $\M$. Thus, by Theorem \ref{thm:changmakkai} and Proposition \ref{propo:DmaxFO} , there exist a finite number of first order formulas $\theta_1(\tuple x, \tuple z) \ldots \theta_n(\tuple x, \tuple z)$ over the empty vocabulary (and in which we can assume -- via variable renaming -- that the variables of $\tuple x \tuple z$ are not being quantified over), such that 
			\[
				\Dmax(R) \models \bigvee_{i=1}^n \exists \tuple z \forall \tuple x (R \tuple x \leftrightarrow \theta_i(\tuple x, \tuple z)).
			\]

			For each $i$, let $\chi_i(\tuple z):=\D(\theta_i(\h,\tuple z))$ be the first order formula obtained from $\D(R)$ by replacing every occurrence $R \tuple t$ of $R$ with $\theta_i(\tuple t, \tuple z)$, so that $\M \models_{(\tuple z:\tuple m)} \chi_i(\tuple z)$ if and only if $\langle M, R\rangle \in \D$ for $R = \{\tuple a : \M \models \theta_i(\tuple a, \tuple m))\}$.

			Then I state that $\D \tuple v$ is equivalent to the $\FO(=\!\!(\cdot), \sqcup)$ formula 
			\begin{equation}
				\bigsqcup_{i=1}^n \exists \tuple z (=\!\!(\tuple z) \wedge \chi_i(\tuple z) \wedge \theta_i(\tuple v, \tuple z))
				\label{eq:tr}
			\end{equation}
			Where $\sqcup$ represents the \emph{Boolean disjunction} of Definition \ref{def:booldisj}, which is definable in terms of constancy atoms.\\

			Indeed, suppose that $\M \models_X \D \tuple v$. Then for $R = X(\tuple v)$ we have that $\langle M, R\rangle \in \D$, and thus by Corollary \ref{coro:maximal} there exists some $R' \supseteq R$ such that $\langle M, R'\rangle \in \Dmax$. Thus, there exists some $i = 1\ldots n$ and a fixed tuple $\tuple m$ of elements of $M$ such that $\M \models_{(\tuple z : \tuple m)} \forall \tuple x(R' \tuple x \leftrightarrow \theta_i(\tuple x, \tuple z))$. Since $\langle M, R'\rangle \in \D$, $\M \models_{(\tuple z: \tuple m)} \chi_i(\tuple z)$. Since $R \subseteq R'$, $\M \models_{(\tuple z : \tuple m)} \forall \tuple x(R \tuple x \rightarrow \theta_i(\tuple x, \tuple z))$ and therefore by Proposition \ref{propo:FOform} $\M \models_{X[\tuple m/\tuple z]} \chi_i(\tuple z) \wedge \theta_i(\tuple v, \tuple z)$.  Moreover, clearly $\M \models_{X[\tuple m/\tuple z]} =\!\!(\tuple z)$, and so $\M \models_{X[\tuple m/\tuple z]} =\!\!(\tuple z) \wedge \chi_i(\tuple z) \wedge \theta_i(\tuple v, \tuple z)$, and finally Equation (\ref{eq:tr}) holds in $X$.\\

			Conversely, suppose that (\ref{eq:tr}) is true in $X$. Then there exists some $i$ and some tuple $\tuple m$ of elements such that $\M \models_{X[\tuple m/\tuple z]} \chi_i(\tuple z)$ and $\M \models_{X[\tuple m/\tuple z]} \theta_i(\tuple v, \tuple z)$. From the first statement, we obtain at once that the set $R' = \{\tuple a: \M \models \theta_i(\tuple a, \tuple m)\}$ is such that $\langle M, R'\rangle \in \D$. From the second one, we obtain that $R = X(\tuple v)$ is contained in $R'$. By the downwards closure of $\D$ we can then conclude that $\langle M, R\rangle \in \D$, that is $\M \models_X \D \tuple v$.
\end{proof}
\section{Conclusions and Further Work}
In this paper, a characterization of downwards closed, relativizable, strongly first order dependencies was found by showing that -- aside from the trivially false one -- these are precisely the ones that are definable in terms of constancy atoms. This is a special case of \textbf{Conjecture 1}, according to which all strongly first order dependencies are definable in terms of constancy atoms \emph{and upwards closed dependencies}; and it is the hope of the author that the techniques developed in this work may be adapted to generalize this result, ideally all the way to a full proof of \textbf{Conjecture 1}. 

A reasonable enough starting point could be to attempt to get rid, or alternatively prove the necessity, of the requirement of relativizability, for instance by showing that all strongly first order dependencies (or at least all downwards closed ones) are relativizable anyway. 

Another possible direction in which the present work can be expanded could be to consider not only logics $\FO(\DD)$ obtained by adding dependency atoms to First Order Logic, but more in general logics $\LL(\DD)$ where $\LL$ can be based on arbitrary choices of connectives and operators (interpreted in Team Semantics). Much like studying the more general notion of safety gave us in this work the tools necessary to prove the above mentioned result about strongly first order dependencies, it is possible that studying the notions of safety and strong first orderness in the more general $\LL(\DD)$ case may give us the tools for solving the $\FO(\DD)$ one; and, moreover, such an investigation would connect the research program to which this work belongs to the related area of the study of generalized quantifiers in Team Semantics \cite{engstrom12,kuusisto2015,engstrom2017dependence,barbero2017some}. 

		\paragraph{Acknowledgements} I thank Fausto Barbero for a great number of insightful comments, suggestions, and corrections (including, among others, his observation about the existence of non-relativizable dependencies mentioned at the end of Section \ref{sec:relativizable}). I also thank an anonymous referee for their many helpful comments and corrections. 
\bibliographystyle{plain}
\bibliography{biblio}
\appendix
\section{The Chang-Makkai Theorem for Non-Unary Relations and Countable Models}
Here I report for convenience a proof of Theorem \ref{thm:changmakkai}. It is no different from the usual proof (via recursive saturation) of the Chang-Makkai Theorem for countable models; but I recall it here anyway to verify that it works even for $k$-ary relations, since this theorem is typically presented only for unary relations. 

	We need to prove the equivalence of the following two statements: 
	\begin{enumerate}
		\item For every countable model $\M$ over the signature $\{R\}$, $|\{R : R \subseteq M^k, \M \models \phi(R)\}| < \aleph_1$; 
		\item There are a finite number of formulas $\theta_1(\tuple x, \tuple z) \ldots \theta_n(\tuple x, \tuple z)$, over the empty vocabulary, such that 
			\begin{equation}
				\phi(R) \models \bigvee_i \exists \tuple z \forall \tuple x (R \tuple x \leftrightarrow \theta_i(\tuple x, \tuple z)).
				\label{eq:def}
			\end{equation}
	\end{enumerate}

	The direction from 2. to 1. is obvious. For the direction from 1. to 2., we reason as follows. Suppose that 2. fails. Then the theory
	\begin{equation}
		\{\phi(R)\} \cup \{\forall \tuple z \lnot \forall \tuple x(R \tuple x \leftrightarrow \theta(\tuple x, \tuple z)) : \theta(\tuple x, \tuple z) \in \FO\},
		\label{eq:undef}
	\end{equation}
	where $\tuple z$ ranges over all tuples of all lengths and $\theta(\tuple x, \tuple z)$ ranges over all first order formulas over the empty signature, is satisfiable. Let $\M = \langle M, R\rangle$ be a countable, recursively saturated model for it. Let $k$ be the arity of $R$, and let $(\tuple t(i) : i \in \mathbb N)$ be an enumeration of all $k$-tuples of elements in $M$. We will show that there exist $2^{\aleph_0}$ distinct $R' \subseteq M^k$ such that $\langle M, R'\rangle$ is isomorphic to $\langle M, R\rangle$; and since $\M \models \phi(R)$, this in particular will show that 1. fails. 

	The idea of the proof is to define two functions $G$ and $H$, sending -- for every $n \in \mathbb N$ -- every function $f: \{0 \ldots n-1\} \rightarrow \{0,1\}$ into some $G(f), H(f) \in \{0 \ldots n-1\} \rightarrow M^k$ such that
	\begin{enumerate}[(i)]
		\item If $f \subseteq g$ then $G(f) \subseteq G(g)$ and $H(f) \subseteq H(g)$; 
		\item For all $f: \{0 \ldots n-1\} \rightarrow \{0,1\}$, 
			\[
				\langle M, G(f)\rangle \equiv \langle M, H(f)\rangle, 
			\]
			where we write $G(f)$ as a shorthand for the tuple of elements $(G(f)(i)_j : i \in 0 \ldots n-1, j \in 1 \ldots k)$ and similarly for $H(f)$; 
		\item If $\dom(f) = \{0 \ldots n-1\}$ for $n$ of the form $3m$ for some $m \in \mathbb N$ then 
			\[
				G(f \cup \{\langle n, 0\rangle\}) = G(f \cup \{\langle n, 1\rangle \}) = G(f) \cup \{\langle n, \tuple t(m)\rangle\};
			\]
		\item If $\dom(f) = \{0 \ldots n-1\}$ for $n$ of the form $3m+1$ then 
			\[
				H(f \cup \{\langle n, 0\rangle\}) = H(f \cup \{\langle n, 1\rangle \}) = H(f) \cup \{\langle n, \tuple t(m)\rangle\};
			\]
		\item If $\dom(f) = \{0, \ldots n-1\}$ for $n$ of the form $3m+2$ then there are two tuples $\tuple b_0 \in R$, $\tuple b_1 \not \in R$ such that
			\begin{align*}
				& G(f \cup \{\langle n, 0\rangle\}) = G(f) \cup \{\langle n, \tuple b_0\rangle\};\\
				& G(f \cup \{\langle n, 1\rangle\}) = G(f) \cup \{\langle n, \tuple b_1\rangle\};\\
				& H(f \cup \{\langle n, 0\rangle\}) = H(f \cup \{\langle n,1\rangle\}).
			\end{align*}
	\end{enumerate}
	Before proving that these $G$, $H$ exist, let us verify that their existence would lead to the required conclusion of there being uncountably many $R'$ isomorphic to $R$. Because of (i), we can apply $G$ and $H$ also to functions $f: \mathbb N \rightarrow \{0,1\}$ in the obvious way by setting 
	\[
		G(f) = \bigcup_{n \in \mathbb N} \{G(f_{|0 \ldots n}) : n \in \mathbb N\},~ H(f) = \bigcup_{n \in \mathbb N} \{H(f_{|0 \ldots n}) : n \in \mathbb N\}
	\]
	and still have that $\langle M, G(f) \rangle \equiv \langle M, H(f)\rangle$ (every first order formula will involve a finite number of constant symbols, so this follows trivially from (ii)). 
	
	For any such $f: \mathbb N \rightarrow \{0,1\}$, since $M$ (without the relation $R$) is a countable model over the empty signature\footnote{We could use the fact that it is recursively saturated and countable, but it is overkill for our needs. In fact, much of this proof could probably be simplified for the case that interests us, but we report it fully anyway.} and since by (iii) and (iv) every element of $M$ appears eventually in both $G(f)$ and $H(f)$,\footnote{In fact, the same element of $M$ will occur multiple times, e.g. as $G(f)(i)_j$ and as $G(f)(i')_{j'}$. But then, by (ii), it is also true that $H(f)(i)_j = H(f)(i')_j$, and vice versa.} it follows that the function $\iota_f$ sending, for each $i \in \mathbb N$ and each $j \in 1 \ldots k$, $G(f)(i)_j$ into $H(f)(i)_j$ is an automorphism of $M$. 

	If the smallest index $n$ such that $f(n) \not = g(n)$ is of the form $3m+2$ then the automorphisms $\iota_f$ and $\iota_g$ differ and map $R$ into two different relations: indeed, $\iota_f$ maps some $\tuple b_0 \in R$ into some tuple $\tuple c = H(f)(n)$, so $\tuple c \in \iota_f[R]$, while $\iota_g$ maps some $\tuple b_1 \not \in R$ into the same $\tuple c$, so $\tuple c \not \in \iota_g[R]$. But there are $2^{\aleph_0}$ many functions every pair of which differs in the first place in an index of this form,\footnote{For any $K \subseteq \mathbb N$ and for all $i \in \mathbb N$, let $f_K(i)$ be $1$ if $i$ is of the form $3m+2$ for some $m \in K$ and let it be $0$ otherwise. Then for all $K_1, K_2 \subseteq \mathbb N$, if $K_1 \not = K_2$ then $f_{K_1}$ and $f_{K_2}$ differ in the first place in such an index.} and therefore there exist $2^{\aleph_0}$ pairwise distinct automorphic images of $R$. Thus there are uncountably many $R$ such that $\M \models \phi(R)$, as required.\\

	It remains to show that the $F$, $G$ described above can be constructed. We proceed by induction on $n \in \mathbb N$, and there are three cases to consider: 
	\begin{itemize}
		\item $n$ is of the form $3m$:\\

			In this case, we extend $G$ by setting -- for all $f$ -- $G(f \cup \{\langle n, 0\rangle\}) = G(f \cup \{\langle n, 1\rangle \}) = G(f) \cup \{\langle n, \tuple t(m)\rangle\}$, and we must show that there exists some tuple $\tuple t' \in M^k$ such that 
			\[
				\langle M, G(f), \tuple t(m)\rangle \equiv \langle M, H(f), \tuple t'\rangle
			\]
			where we know, by induction hypothesis, that $\langle M, G(f)\rangle \equiv \langle M, H(f)\rangle$. This is clearly the case, because recursively saturated models are countably homogeneous.\footnote{Or, more simply, because $\M$ is infinite and has no relation symbol in the signature aside from $R$.  Using the recursive saturation of $\langle M, R\rangle$ here is actually really overkill; but we need it anyway to deal with the $3m+2$ case.}

			Then extend $H$ by setting $H(f \cup \{\langle n,0\rangle\}) = H(f \cup \{\langle n, 1\rangle\}) = H(f) \cup \{\langle n, \tuple t'\rangle\}$: the conditions about $F$ and $G$ are still satisfied. 
		\item $n$ is of the form $3m+1$:\\

			In this case, we extend $H$ as $H(f \cup \{\langle n, 0\rangle\}) = H(f \cup \{\langle n, 1\rangle \}) = H(f) \cup \{\langle n, \tuple t(m)\rangle\}$, and we must show that there exists some tuple $\tuple t' \in M^k$ such that 
			\[
				\langle M, G(f), \tuple t'\rangle \equiv \langle M, H(f), \tuple t(m)\rangle
			\]
			Again, this follows easily from countable homogeneity (or, more simply, from the fact that $M$ has no relation symbols in the signature).

			Then extend $G$ by setting $G(f \cup \{\langle n,0\rangle\}) = G(f \cup \{\langle n, 1\rangle\}) = G(f) \cup \{\langle n, \tuple t'\rangle\}$: the conditions about $F$ and $G$ are still satisfied. 

		\item $n$ is of the form $3m+2$:\\
		
			We need to find tuples $\tuple b_0 \in R$, $\tuple b_1 \not \in R$, $\tuple c$ such that 
			\[
				\langle M, G(f), \tuple b_0\rangle \equiv \langle M, H(f), \tuple c\rangle
			\]
			and 
			\[
				\langle M, G(f), \tuple b_1\rangle \equiv \langle M, H(f), \tuple c\rangle. 
			\]
			
			To do this, we first find $\tuple b_0 \in R$, $\tuple b_1 \not \in R$ such that $\langle M, G(f), \tuple b_0\rangle \equiv \langle M, G(f), \tuple b_1\rangle$.  Suppose that no such $\tuple b_0, \tuple b_1$ exist: then for all $\tuple b_0 \in R$, consider the recursive type 
			\[
				\{\lnot R \tuple x\} \cup \{
					\theta(\tuple b_0) \leftrightarrow \theta(\tuple x) : \theta(\tuple x) \in \FO
				\}
			\]
			where, in the above expression, $\theta$ ranges over all first order formulas (over the empty signature) with parameters in $G(f)$.  The above expression must be finitely unsatisfiable in $\langle M, R\rangle$; otherwise, by the recursive saturation of $(M,R)$, we would have that a $\tuple b_1$ exists as required. So, for each $\tuple b_0 \in R$ there exists a $\sigma_{\tuple b_0}(\tuple x)$ such that $\M \models \sigma_{\tuple b_0}(\tuple b_0)$ and $\M \models \forall \tuple x (\sigma_{\tuple b_0}(\tuple x) \rightarrow R \tuple x)$.\footnote{Indeed, let $\Gamma_0 = \{\lnot R \tuple x\} \cup \{\theta_i(\tuple b_0) \leftrightarrow \theta_i(\tuple x) : i = 1 \ldots l\}$ be unsatisfiable in $\langle M, R\rangle$. Then define $\sigma_{\tuple b_0}(\tuple x)$ as $\bigwedge \{\theta_i(\tuple x) : i = 1 \ldots l, \M \models \theta_i(\tuple b_0)\} \wedge \bigwedge\{\lnot \theta_i(\tuple x) : i = 1 \ldots l, \M \not \models \theta_i(\tuple b_0)\}$. Clearly $\M \models \sigma_{\tuple b_0}(\tuple b_0)$; and moreover, if $\M \models \sigma_{\tuple b_0}(\tuple c)$ we have that $\M \models \theta_i(\tuple b_0) \leftrightarrow \theta_i(\tuple c)$ for all $i = 1 \ldots l$, and therefore by the unsatisfiability of $\Gamma_0$ we have that $\M \models R \tuple c$.}  Then consider the theory 
			\[
				\{R \tuple x\} \cup \{ \forall \tuple y(\sigma(\tuple y) \rightarrow R \tuple y) \rightarrow \lnot \sigma(\tuple x) : \sigma(\tuple x) \in \FO\}
			\]
			where, again, $\sigma(\tuple x)$ ranges over all first order formulas over the empty signature with parameters in $G(f)$. This must also be finitely unsatisfiable given the theory of $M$, since otherwise by recursive saturation there would exist some $\tuple b \in R$ such that for no formula $\sigma$ we have $\M \models \forall \tuple y (\sigma(\tuple y) \rightarrow R \tuple y)$ and $\M \models \sigma(\tuple b)$, and we just saw that $\sigma_{{\tuple b}}$ is such a formula. 

			Therefore, there exists some finite number of $\sigma$s, which we can write as $\sigma_1(\tuple x, \tuple a) \ldots \sigma_q(\tuple x, \tuple a)$ for some tuple $\tuple a$ of parameters in $G(f)$, such that $\M \models \forall \tuple y(\sigma_i(\tuple y, \tuple a) \rightarrow R \tuple y)$ for all $i$ and such that $\M \models \forall \tuple x (R \tuple x \rightarrow \bigvee_i \sigma_i (\tuple x, \tuple a))$. But then $\rho(\tuple x, \tuple a) := \bigvee_i \sigma_i(\tuple x, \tuple a)$ defines $R$ in $M$ in the sense that $\M \models \forall \tuple x (R \tuple x \leftrightarrow \rho(\tuple x, \tuple a))$, which is impossible since $\langle M, R\rangle$ is a model of (\ref{eq:undef}). 
			
			Thus, it is possible to find $\tuple b_0$, $\tuple b_1$ such that $\langle M, G(f), \tuple b_0\rangle \equiv \langle M, G(f), \tuple b_1\rangle$.\\

			Now let us find -- using, once more, the countable homogeneity of $M$ -- some $\tuple c$ such that $\langle M, G(f), \tuple b_0\rangle \equiv \langle M, H(f), \tuple c\rangle$. Since $\langle M, G(f), \tuple b_0\rangle \equiv \langle M, G(f), \tuple b_1\rangle$, we then also have that $\langle M, G(f), \tuple b_1\rangle \equiv \langle M, H(f), \tuple c\rangle$. Finally, extend $G$ by setting $G(f \cup \{\langle n, 0\rangle\}) = G(f) \cup \{\langle n, \tuple b_0\rangle\}$ and $G(f \cup \{\langle n, 1\rangle\}) = G(f) \cup \{\langle n, \tuple b_1\rangle\}$, and extend $H$ by setting $H(f \cup \{\langle n, 0\rangle\}) = H(f \cup \{\langle n,1\rangle\}) = H(f) \cup \{\langle n, \tuple c\rangle\}$. The conditions are all still satisfied, and this concludes the proof. 
	\end{itemize}
\end{document}